\documentclass[a4, 12pt]{amsart}
\usepackage{amsthm,enumerate,graphicx,color}
\usepackage{amssymb}
\usepackage{amstext}
\usepackage{amsmath}
\usepackage{amscd}
\usepackage{latexsym}
\usepackage{amsfonts}
\theoremstyle{plain}
\newtheorem{thm}{Theorem}[section]
\newtheorem*{thm*}{Theorem}
\newtheorem*{cor*}{Corollary}

\newtheorem{prop}[thm]{Proposition}
\newtheorem{lem}[thm]{Lemma}
\newtheorem{cor}[thm]{Corollary}
\newtheorem{claim}{Claim}
\newtheorem*{claim*}{Claim}

\theoremstyle{definition}
\newtheorem{defn}[thm]{Definition}
\newtheorem{ex}[thm]{Example}
\newtheorem{rem}[thm]{Remark}

\theoremstyle{remark}

\numberwithin{equation}{thm}
\def\Hom{\mathrm{Hom}}

\def\GCD{\mathrm{GCD}}

\def\Coker{\mathrm{Coker}}

\def\Im{\mathrm{Im}}

\def\rank{\operatorname{rank}}

\def\e{\mathrm{e}}
\def\m{\mathfrak m}

\def\q{\mathfrak q}

\def\N{\mathbb{N}}
\def\Z{\mathbb{Z}}

\newcommand{\rma}{\mathrm{a}}

\newcommand{\rme}{\mathrm{e}}

\newcommand{\rmr}{\mathrm{r}}

\newcommand{\rmv}{\mathrm{v}}

\newcommand{\rmA}{\mathrm{A}}

\newcommand{\rmD}{\mathrm{D}}
\newcommand{\rmE}{\mathrm{E}}

\newcommand{\rmI}{\mathrm{I}}

\newcommand{\rmK}{\mathrm{K}}

\newcommand{\rmR}{\mathrm{R}}

\newcommand{\calC}{\mathcal{C}}

\newcommand{\calF}{\mathcal{F}}

\newcommand{\calM}{\mathcal{M}}

\newcommand{\calS}{\mathcal{S}}

\newcommand{\calX}{\mathcal{X}}
\newcommand{\calY}{\mathcal{Y}}

\newcommand{\fkm}{\mathfrak{m}}
\newcommand{\fkn}{\mathfrak{n}}

\newcommand{\fkp}{\mathfrak{p}}

\newcommand{\mapright}[1]{%
\smash{\mathop{%
\hbox to 1cm{\rightarrowfill}}\limits^{#1}}}

\newcommand{\mapleft}[1]{%
\smash{\mathop{%
\hbox to 1cm{\leftarrowfill}}\limits_{#1}}}

\newcommand{\mapdown}[1]{\Big\downarrow
\llap{$\vcenter{\hbox{$\scriptstyle#1\,$}}$ }}

\def\Syz{\mathrm{Syz}}

\def\gr{\operatorname{gr}}

\def\X{{\mathcal X}}

\def\ol{\overline}

\begin{document}

\setlength{\baselineskip}{15pt}
\title{Ulrich ideals and modules}
\pagestyle{plain}
\author{Shiro Goto, Kazuho Ozeki, Ryo Takahashi, Kei-ichi Watanabe, Ken-ichi Yoshida}
\address{S. Goto: Department of Mathematics, School of Science and Technology, Meiji University, 1-1-1 Higashimita, Tama-ku, Kawasaki 214-8571, Japan}
\email{goto@math.meiji.ac.jp}
\address{K. Ozeki: Department of Mathematical Science, Faculty of Science, Yamaguchi University, 1677-1 Yoshida, Yamaguchi 853-8512, Japan}
\email{ozeki@yamaguchi-u.ac.jp}
\address{R. Takahashi: Graduate School of Mathematics, Nagoya University, Furocho, Chikusaku, Nagoya 464-8602, Japan/Department of Mathematics, University of Nebraska, Lincoln, NE 68588-0130, USA}
\email{takahashi@math.nagoya-u.ac.jp}
\urladdr{http://www.math.nagoya-u.ac.jp/~takahashi/}
\address{K.-i. Watanabe and K. Yoshida: Department of Mathematics, College of Humanities and Sciences, Nihon University, 3-25-40 Sakurajosui, Setagaya-Ku, Tokyo 156-8550, Japan}
\email{watanabe@math.chs.nihon-u.ac.jp}
\email{yoshida@math.chs.nihon-u.ac.jp}
\thanks{This work was partially supported by JSPS Grant-in-Aid for Scientific Research (C) 20540050/22540047/22540054/23540059, JSPS Grant-in-Aid for Young Scientists (B) 22740008/22740026 and by JSPS Postdoctoral Fellowships for Research Abroad}
\begin{abstract}
In this paper we study Ulrich ideals of and Ulrich modules over Cohen--Macaulay local rings from various points of view. We determine the structure of minimal free resolutions of Ulrich modules and their associated graded modules, and classify Ulrich ideals of numerical semigroup rings and rings of finite CM-representation type.
\end{abstract}
\maketitle
\tableofcontents


\section{Introduction}\label{intro}

The purpose of this paper is to report the study of Ulrich ideals and modules with a generalized form. We shall explore their structure  and  establish, for given Cohen--Macaulay local rings, the ubiquity of these kinds of ideals and modules.

Ulrich modules with respect to maximal ideals in our sense, that is MGMCM (\underline{m}aximally \underline{g}enerated \underline{m}aximal \underline{C}ohen--\underline{M}acaulay) modules were introduced by \cite{U,BHU} and have been closely explored in connection to the representation theory of rings.
Our motivation has started, with a rather different view-point, from the naive question of why the theory of MGMCM modules works so well. We actually had an occasion \cite{G} to make a heavy use of it and wanted to know the reason.

To state the main results, let us begin with the definition of Ulrich ideals and modules. Throughout this paper, let $A$ be a Cohen--Macaulay local ring with  maximal ideal $\fkm$ and $d = \dim A \geq 0$. Let $I$ be an $\fkm$--primary ideal of $A$ and  let $$\operatorname{gr}_I(A) = \bigoplus_{n \ge 0}I^n/I^{n+1}$$ be the associated graded ring of $I$. For simplicity, we assume that $I$ contains a parameter ideal $Q = (a_1, a_2, \ldots, a_d)$ of $A$ as a reduction. Notice that this condition is automatically satisfied, if the residue class field $A/\fkm$ of $A$ is infinite, or if $A$ is analytically irreducible and $\dim A = 1$. Let $\rma (\operatorname{gr}_I(A))$ denote the $a$--invariant of $\operatorname{gr}_I(A)$ (\cite[Definition 3.1.4]{GW}).

\begin{defn}\label{1.1}
We say that $I$ is an Ulrich ideal of $A$ if it satisfies the following.
\begin{enumerate}
\item[$(1)$] $\operatorname{gr}_I(A)$ is a Cohen--Macaulay  ring with $\rma (\operatorname{gr}_I(A)) \le 1-d$.
\item[$(2)$] $I/I^2$ is a free $A/I$--module.
\end{enumerate}
\end{defn}
 
Condition (1) of Definition \ref{1.1} is equivalent to saying that $I^2 = QI$ (\cite[Remark 3.1.6]{GW}). Hence every parameter ideal is an Ulrich ideal.  When $I = \fkm$, Condition (2) is naturally satisfied and Condition (1) is equivalent to saying that the Cohen--Macaulay local ring $A$ possesses maximal embedding dimension in the sense of J. Sally \cite{S1}, namely the equality
$\rmv(A) = \rme(A) + d -1$
holds true, where $\rmv (A)$ (resp. $\rme (A)$) denotes the embedding dimension    of $A$, that is the minimal number $\mu_A(\fkm)$ of generators of the maximal ideal $\fkm$ (resp. the multiplicity $\rme_\fkm^0(A)$ of $A$ with respect to $\fkm$).

\begin{defn}\label{Umodule}
Let $M$ be a finitely generated $A$-module. Then we say that $M$ is an Ulrich $A$--module with respect to $I$, if the following conditions are satisfied. 
\begin{itemize}
\item[(1)] $M$ is a maximal Cohen--Macaulay $A$-module, that is $\operatorname{depth}_AM = d$.
(Hence the zero module is not maximal Cohen--Macaulay in our sense).
\item[(2)] $\e_I^0(M)=\ell_A(M/IM)$.
\item[(3)] $M/IM$ is $A/I$-free.
\end{itemize}
Here $\e_I^0(M)$ denotes the multiplicity of $M$ with respect to $I$ and $\ell_A(M/IM)$ denotes the length of the $A$--module $M/IM$.
\end{defn}

When $M$ is a maximal Cohen--Macaulay $A$--module, we have 
$$\rme_I^0(M) = \rme_Q^0(M) = \ell_A(M/QM) \ge \ell_A(M/IM),$$
so that Condition (2) of Definition \ref{Umodule} is equivalent to saying that $IM = QM$. Therefore, if $I = \fkm$, Condition (2) is equivalent to saying that $\rme_\fkm^0(M) = \mu_A(M)$, that is $M$ is maximally generated in the sense of \cite{BHU}. Similarly to the case of Ulrich ideals, every Ulrich $A$--module with respect to a parameter ideal is free.

Our purpose is to explore the structure of Ulrich ideals and modules in the above sense and investigate how many Ulrich ideals and modules exist over a given Cohen--Macaulay local ring $A$.

This paper consists of nine sections. In Sections \ref{ulid} and \ref{ulmd} we will summarize basic properties of Ulrich ideals and modules. Typical examples we keep in mind shall be given. Higher syzygy modules of Ulrich ideals are Ulrich modules (Theorem \ref{3.2}), which we will prove in Section \ref{ulmd}. Several results of this paper are proven by induction on $d = \dim A$. In Section \ref{uinsr} we shall closely explain the induction technique, which is due to and dates back to W. V. Vasconcelos \cite{V}.

The converse of Theorem \ref{3.2} is also true. We will show in Section \ref{reluium} that our ideal $I$ is Ulrich, once the higher syzygy modules of $I$ are Ulrich $A$--modules with respect to $I$ (Theorem \ref{4.1}). We will discuss in Section \ref{dual} the problem of when the canonical dual of Ulrich $A$--modules are again Ulrich.

It seems natural and interesting to ask if how many Ulrich ideals which are not parameter ideals are contained in a given Cohen--Macaulay local ring. The research about this question is still in progress and we have no definitive answer. In Section \ref{uinsr} we shall study the case where $A$ is a numerical semigroup ring over a field $k$, that is 
\[
A = k[[t^{a_1}, t^{a_2}, \ldots, t^{a_\ell}]] \ \ \subseteq k[[t]],
\]
where $0 < a_1, a_2, \ldots, a_\ell \in \mathbb{Z}$~$(\ell > 0)$ with 
$\operatorname{GCD}(a_1, a_2, \ldots, a_\ell) = 1$ and 
$k[[t]]$ denotes the formal power series ring. 
We also restrict our attention to the set ${\mathcal X}^g_A$ of Ulrich ideals $I$ of $A$ which are generated by monomials in $t$ but not parameter ideals, namely $\mu_A(I) > 1$. 
Then  ${\mathcal X}^g_A$ is a finite set (Theorem \ref{6.1}). 
We will show in Section \ref{uinsr} the following structure theorem of those Ulrich ideals also, 
when $A$ is a Gorenstein ring, that is the case where the semigroup 
\[
H = \langle a_1, a_2, \ldots, a_\ell \rangle = \left\{\sum_{i = 1}^\ell c_i a_i\,\Bigg|\,0 \le c_i \in \mathbb{Z} \right\}
\]
generated by the integers $a_i's$ is symmetric.
(Recall that a numerical semigroup $H$ is called symmetric provided that for all $i$ with $0\le i\le c-1$ one has $i\in H$ if and only if $c-i-1\notin H$, where $c=\max\{h\in\N\mid h-1\notin H\}$ is the conductor of $H$.)

\begin{thm}[\,=\,Theorem \ref{6.3}]
Suppose that $A$ is a Gorenstein ring and let $I$ be an ideal of $A$. Then the following conditions are equivalent.
\begin{enumerate}
\item[$(1)$] $I \in {\mathcal X}_A^g$.
\item[$(2)$] There exist elements $a, b \in H$ such that 
\begin{enumerate}
\item[$\mathrm{(i)}$] $a < b$ and $I = (t^a, t^b)$,
\item[$\mathrm{(ii)}$] $b-a\not\in H$ and $2(b-a)\in H$,
\item[$\mathrm{(iii)}$] the numerical semigroup $H_1 = H + \left<b-a\right>$ is symmetric, and
\item[$\mathrm{(iv)}$] $a = \min \{h \in H \mid (b-a)+ h \in H\}$.
\end{enumerate}
\end{enumerate}
\end{thm}

As a consequence, we show that for given integers $1 < a < b$ with $\operatorname{GCD}(a,b)= 1$, the numerical semigroup ring $A = k[[t^a,t^b]]$ contains at least one Ulrich ideal generated by monomials in $t$ that is not a parameter ideal if and only if $a$ or $b$ is even (Theorem \ref{H=(a,b)}).

Section \ref{minfree} is devoted to the analysis of minimal free resolutions of Ulrich 
ideals. Let $I$ be an Ulrich ideal of a $d$-dimensional Cohen--Macaulay local ring $A$ and put $n = \mu_A(I)$. Let
$$\mathbb{F}_{\bullet} : \cdots \to F_i \overset{\partial_i}{\to} F_{i-1} \to \cdots \to F_1 \overset{\partial_1}{\to} F_0 = A \overset{\varepsilon}{\to} A/I \to 0$$
be a minimal free resolution of the $A$--module $A/I$ and put $\beta_i = \operatorname{rank}_AF_i$.  We then have the following.

\begin{thm}[\,=\,Theorem \ref{7.1}]\label{betti}
One has $A/I \otimes_A \partial_i=0$ for all $i \geq 1$, and
$$
\beta_i= \left\{
\begin{array}{ll}
(n-d)^{i-d}{\cdot}(n-d+1)^d & (d \le i),\\
\binom{d}{i}+(n-d){\cdot}\beta_{i-1} & (1 \leq i \leq d),\\
1 & (i=0).
\end{array}
\right.
$$
Hence $\beta_i=\binom{d}{i}+(n-d){\cdot}\beta_{i-1}$ for all $i \geq 1$.
\end{thm}

What Theorem \ref{betti} says is that, thanks to the exact sequence $0 \to Q \to I \to (A/I)^{n-d} \to 0$, a minimal free resolution of the Ulrich ideal $I$ is isomorphic to the resolution induced from  the direct sum of $n-d$ copies of $\mathbb{F}_{\bullet}$ and the minimal free resolution of $Q = (a_1, a_2, \ldots, a_d)$, that is the truncation 
$${\mathbb{L}}_{\bullet} : 0 \to K_d \to K_{d-1} \to \cdots \to K_1 \to Q \to 0$$
of the Koszul complex ${\mathbb{K}}_{\bullet}(a_1, a_2, \ldots, a_d; A)$ generated by the $A$--regular sequence $a_1, a_2, \ldots, a_d$. As consequences, we get that $\mathbb{F}_\bullet$ is eventually periodic, if $A$ is a Gorenstein ring and that for Ulrich ideals $I$ and $J$ which are not parameter ideals, one has $I = J$, once 
$$\operatorname{Syz}_A^i(A/I) \cong \operatorname{Syz}_A^i(A/J)$$
for some $i \ge 0$, where 
$\operatorname{Syz}_A^i(A/I)$ and $\operatorname{Syz}_A^i(A/J)$ denote the $i$-th syzygy modules of $A/I$ and $A/J$ in their minimal free resolutions, respectively. The latter result eventually yields the following, which we shall prove in Section \ref{minfree}.
Recall here that a Cohen--Macaulay local ring is said to be of finite CM--representation type if there exist only finitely many nonisomorphic indecomposable maximal Cohen--Macaulay modules.

\begin{thm}[\,=\,Therem \ref{7.7}]
If $A$ is a Cohen--Macaulay local ring of finite CM--representation type, then $A$ contains only finitely many Ulrich ideals which are not parameters.
\end{thm}

In Section \ref{linear} we study the linearity of a minimal free resolution of an associated graded module of an Ulrich module.
We prove the following result.

\begin{thm}[\,=\,Theorem \ref{9.5}]
Let $A$ be a Cohen--Macaulay local ring.
Let $I$ be an Ulrich ideal of $A$ and $M$ an Ulrich $A$-module with respect to $I$.
Then the associated graded module $\gr_I(M)$ of $M$ has a minimal free resolution
$$
\cdots \to \bigoplus^{r_i}\gr_I(A)(-i) \to \cdots \to \bigoplus^{r_1}\gr_I(A)(-1) \to \bigoplus^{r_0}\gr_I(A) \to \gr_I(M) \to 0
$$
as a graded $\gr_I(A)$-module, where $r_i$ is the $i$-th Betti number of $M$ for $i \geq 0$.
\end{thm}


In the final Section \ref{onefcm} we determine all the Ulrich ideals of $A$ when $A$ is a $1$-dimensional Gorenstein local ring of finite CM-representation type.
We prove the following theorem by using techniques from the representation theory of maximal Cohen--Macaulay modules.
In the forthcoming paper \cite{GOTWY}, we will prove a similar result for $2$-dimensional 
Gorenstein rational singularities (Gorenstein local rings of finite CM-representation type). 

\begin{thm}[\,=\,Theorem \ref{8.9}]
Let $A$ be a $1$-dimensional Gorenstein complete equicharacteristic local ring with algebraically closed residue field $k$ of characteristic zero.
Suppose that $A$ has finite CM-representation type.
Then $A$ is a simple ADE-singularity $k[[x,y]]/(f)$, and the set of Ulrich ideals of $A$ is as follows:
\begin{enumerate}
\item[$(\rmA_n)$\ ]
$\begin{cases}
\{(x,y),(x,y^2),\dots,(x,y^{\frac{n}{2}})\} & \text{if $n$ is even},\\
\{(x,y),(x,y^2),\dots,(x,y^{\frac{n-1}{2}}),(x,y^{\frac{n+1}{2}})\} & \text{if $n$ is odd}.
\end{cases}$
\item[$(\rmD_n)$\ ]
$\begin{cases}
\{(x^2,y),(x+\sqrt{-1}y^\frac{n-2}{2},y^\frac{n}{2}),(x-\sqrt{-1}y^\frac{n-2}{2},y^\frac{n}{2})\} & \text{if $n$ is even},\\
\{(x^2,y),(x,y^\frac{n-1}{2})\} & \text{if $n$ is odd}.
\end{cases}$
\item[$(\rmE_6)$\ ]
$\{(x,y^2)\}$.
\item[$(\rmE_7)$\ ]
$\{(x,y^3)\}$.
\item[$(\rmE_8)$\ ]
$\emptyset$.
\end{enumerate}
\end{thm}

Unless otherwise specified, throughout this paper, let $(A,\fkm)$ be a Cohen--Macaulay local ring of dimension $d \ge 0$ and let $I$ be an $\fkm$--primary ideal of $A$ which contains a parameter ideal $Q = (a_1, a_2, \ldots, a_d)$ as a reduction. We put $n =\mu_A(I)$, the number of elements in a minimal system of generators of $I$.


\section{Ulrich ideals}\label{ulid}
The purpose of this section is to summarize basic properties of Ulrich ideals. To begin with, let us recall the definition.

\begin{defn}\label{2.1}
We say that $I$ is an Ulrich ideal of $A$, if $I^2=QI$ and the $A/I$--module $I/I^2$ is free.
\end{defn}

We note the following.

\begin{ex}\label{2.2}
Let $R$ be a Cohen--Macaulay local ring with maximal ideal $\fkn$ and dimension $d \ge 0$. Let $F=R^r$ with $r > 0$ and let  $A=R \ltimes F$ be the idealization of $F$ over $R$. Then $A$ is a Cohen--Macaulay local ring with maximal ideal $\fkm = \fkn \times F$ and $\dim A = d$. 
Let $\q$ be an arbitrary parameter ideal of $R$ and put $Q=\q A$. Then the ideal  $I=\q \times F$ of $A$ contains the parameter ideal $Q$  as a reduction. We actually have $I^2 = QI$ and $I/I^2$ is $A/I$--free, so that $I$ is an Ulrich ideal of $A$ with $n = \mu_A(I) = r + d > d$. Hence the local ring $A$ contains infinitely many Ulrich ideals which are not parameters.
\end{ex}

\begin{proof}
It is routine to check that $I^2 = QI$, while $I/I^2 = \q/\q^2 \times F/\q F$ is a free module over $A/I = R/\q$.
\end{proof}



\begin{lem} \label{2.3}
Suppose that $I^2 =QI$.  
Then$:$
\begin{enumerate}[\rm(1)]
\item $\e_{I}^0(A) \le (\mu(I)-d+1)\ell_A(A/I)$ holds true. 
\item The following conditions for $I$ are equivalent$:$
\begin{enumerate}[\rm(a)]
\item Equality holds in $(1)$. 
\item $I$ is an Ulrich ideal. 
\item $I/Q$ is a free $A/I$-module. 
\end{enumerate}
\end{enumerate} 
\end{lem}

\begin{proof}
(1)
Since $Q$ is generated by an $A$-sequence, 
$\ell_A(Q/I^2) = \ell_A(Q/QI)=d \cdot \ell_A(A/I)$. 
Hence
$\ell_A(I/I^2) = \ell_A(A/Q) + \ell_A(Q/QI)-\ell_A(A/I)
=\e_{I}^0(A)+ (d-1) \cdot \ell_A(A/I)$.
\par 
On the other hand, since there exists a natural surjection 
$(A/I)^{\mu_A(I)} \to I/I^2$, we have 
$\ell_A(I/I^2) \le \mu_A(I)\cdot \ell_A(A/I)$.
Thus we obtain the required inequality. 
\par 
(2) $(a)\Leftrightarrow (b):$
Equality holds true if and only if the surjection as above
$(A/I)^{\mu_A(I)}\to I/I^2$ is an isomorphism, that is, 
$I/I^2$ is a free $A/I$-module. 
This is equivalent to saying that $I$ is an Ulrich ideal. 
\par
 $(b)\Leftrightarrow (c):$ 
We look at the canonical exact sequence
$0 \to Q/I^2 \overset{\varphi}{\to} I/I^2 \to I/Q \to 0$
of $A/I$--modules. Then, since $I^2 = QI$, $Q/I^2 = A/I \otimes_{A/Q}Q/Q^2$ is $A/I$--free, whence $I/I^2$ 
is $A/I$--free if $I/Q$ is $A/I$--free. Conversely, if $I/I^2$ is $A/I$--free, then the $A/I$-module 
$I/Q$ has projective dimension at most $1$. As $A/I$ is Artinian, $I/Q$ is $A/I$--free.
\end{proof}

We need the following result in Section \ref{ulmd}.

\begin{prop}\label{2.4}
Suppose that the residue class field $A/\fkm$ of $A$ is infinite. Then the following conditions are equivalent.
\begin{enumerate}
\item[$(1)$] $I$ is an Ulrich ideal of $A$.
\item[$(2)$] For every minimal reduction $\q$ of $I$, $I^2 \subseteq \q$ and the $A/I$--module $I/\q$ is free.
\end{enumerate}
\end{prop}

\begin{proof} Thanks to Lemma \ref{2.3}, we have only to show $(2) \Rightarrow (1)$.  We may assume $n > d > 0$. Let us  choose elements $x_1, x_2, \ldots, x_n$ of $I$ so that $I = (x_1, x_2, \ldots, x_n)$ and the ideal $(x_{i_1}, x_{i_2}, \ldots, x_{i_d})$ is a reduction of $I$ for every choice of integers $1 \le i_1 < i_2 < \ldots < i_d \le n$.
(Since $A/\fkm$ is infinite, there exists a reduction of $I$.)

We will firstly show that $I/I^2$ is $A/I$--free. Let $c_1, c_2, \ldots, c_n \in A$ and assume that $\sum_{i=1}^nc_ix_i\in I^2$. Let $1 \le i \le n$ and choose a subset $\Lambda \subseteq \{1, 2, \ldots, n \}$ so that $\sharp\Lambda=d$ and $i \not\in \Lambda$. We put $\q = (x_j \mid j \in \Lambda)$. Then, because $\mu_A(I/\q) = n -d$ and $I/\q = (\overline{x_j} \mid j \not\in \Lambda)$, $\{\overline{x_j}\}_{j \not\in \Lambda}$ form an $A/I$--free basis of $I/\q$, where $\overline{x_j}$ denotes the image of $x_j$ in $I/\q$. Therefore, $c_j \in I$ for all $j \not\in \Lambda$, because $\sum_{j \not\in \Lambda}c_j\overline{x_j}= 0$ in $I/\q$; in particular, $c_i \in I$. Hence $I/I^2$ is $A/I$--free.

Let $\q = (x_1, x_2, \ldots, x_d)$ and let $y \in I^2 \subseteq \q$. We write $y = \sum_{i=1}^dc_ix_i$ with $c_i \in A$. Then, since $\sum_{i=1}^dc_i\overline{x_i} = 0$ in $I/I^2$, we get $c_i \in I$ for all $1 \le i \le d$, because the images $\overline{x_i}$ of $x_i$~($1 \le i \le n)$ in $I/I^2$ form an $A/I$--free basis of $I/I^2$. Thus $y \in \q I$, so that $I^2 = \q I$ and $I$ is an Ulrich ideal of $A$.
\end{proof}

\begin{rem}\label{2.5}
Even though $I^2 \subseteq Q$ and $I/Q$ is $A/I$--free for some minimal reduction $Q$, the ideal $I$ is not necessarily an Ulrich ideal, as the following example shows. Let $A = k[[t^4,t^5,t^6]] \subseteq k[[t]]$, where $k[[t]]$ denotes the formal power series ring over a field $k$. Let $ I = (t^4, t^5)$ and $Q=(t^4)$. Then $I^4 = Q I^3$ but $I^3 \ne Q I^2$, while $I^2 \subseteq Q$ and $I/Q \cong A/I$. 
\end{rem}

For each Cohen--Macaulay $A$--module $M$ of dimension $s$, we put 
$$\rmr_A(M) = \ell_A(\operatorname{Ext}_A^s(A/\fkm,M))$$ and call it the Cohen--Macaulay type of $M$. 
Suppose that $A$ is a Gorenstein ring. Then $I$ is said to be \textit{good}, if $I^2 = QI$ and $Q : I = I$ (\cite{GIW}).
With this notation Ulrich ideals of a Gorenstein ring are characterized in the following way.

\begin{cor}\label{2.6}
Assume that $I$ is not a parameter ideal and put $n = \mu_A(I) > d$.
\begin{enumerate}
\item[$(\mathrm{a})$] If $I$ is an Ulrich ideal of $A$, then $(\mathrm{i})$ $Q : I = I$ and $(\mathrm{ii})$ $(n-d){\cdot}\rmr (A/I) \le \rmr (A)$, whence $n \le \rmr (A) + d$.
\item[$(\mathrm{b})$] Suppose that $A$ is a Gorenstein ring. Then the following are equivalent.
\begin{enumerate}
\item[$(1)$] $I$ is an Ulrich ideal of $A$.
\item[$(2)$] $I$ is a good ideal of $A$ and $\mu_A(I) = d+ 1$.
\item[$(3)$] $I$ is a good ideal of $A$ and $A/I$ is a Gorenstein ring.
\end{enumerate}
\end{enumerate}
\end{cor}

\begin{proof}
Since $I/Q \cong (A/I)^{n-d}$ and $n > d$, assertion (i) is clear. This isomorphism also shows $(n-d){\cdot}\rmr (A/I) = \rmr_A(I/Q) \le \rmr (A/Q) = \rmr (A),$ which is assertion (ii). 

We now suppose that $A$ is a Gorenstein ring. If $I$
 is an Ulrich ideal of $A$, then by assertion (i) $I$ is a good ideal of $A$ and  $(n-d){\cdot}\rmr (A/I)=1$, so that $n = d+1$ and $A/I$ is a Gorenstein ring. Conversely, assume that $I$ is a good ideal of $A$. If $n = d+1$, then since $I/Q$ is a cyclic $A$--module with $Q:I = I$, we readily get $I/Q \cong A/I$, whence $I$ is an Ulrich ideal of $A$ by Lemma \ref{2.3}. If $A/I$ is a Gorenstein ring, then $I/Q$ is a faithful $A/I$--module.
Since in general a finitely generated faithful module over an Artinian Gorenstein local ring is free, $I/Q$ is $A/I$-free.
\end{proof}

We close this section with the following examples.

\begin{ex}\label{2.7}
Let $k$ be a field.
\begin{enumerate}
\item[$(1)$] Let $A = k[[t^4, t^6, t^{4\ell - 1}]]$~($ \ell \ge 2$). Then $I = (t^4, t^6)$ is an Ulrich ideal of $A$ containing $Q = (t^4)$ as a reduction.
\item[$(2)$]
Let $q, d \in \mathbb{Z}$ such that $d \ge 1$ and $r\ge1$.
Let $R = k[[X_1, \ldots, X_d, X_{d+1}]]$ be the formal power series ring and let $A = R/(X_1^2 + \cdots + X_d^2 + X_{d+1}^{2r})$.
Let $x_i$ be the image of $X_i$ in $A$ and put $I = (x_1, \ldots, x_d, x_{d+1}^r)$.
Then $I$ is an Ulrich ideal of $A$ with $\mu_A(I) = d + 1$. 
\item[$(3)$] Let $K/k~(K \ne k)$ be a finite extension of fields and assume that there are no proper intermediate fields between $K$ and $k$. Let $V=K[[t]]$ be the formal power series ring over $K$ and put $A = k+tK[[t]]~\subseteq V$. Then the ring $A$ contains a unique  Ulrich ideal, that is $\fkm = tV$, except parameter ideals.
\end{enumerate}
\end{ex}

\begin{proof}
(1) Let us identify $A = k[[X, Y, Z]]/(X^3-Y^2, Z^2 - X^{2\ell -2}Y)$, where $k[[X,Y,Z]]$ denotes the formal power series ring. We then have $I^2 = QI + (t^{12}) = QI$, while $Q : I = I$. In fact, since $A/Q = k[[Y,Z]]/(Y^2, Z^2)$ and $(0) : y = (y)$ in  $k[[Y,Z]]/(Y^2, Z^2)$ where $y$ denotes the image of $Y$ in $k[[Y,Z]]/(Y^2, Z^2)$, the equality $Q : I = I$ follows.

(2) Let $Q = (x_1, x_2, \ldots, x_d)$ (resp. $Q = (x_2, \ldots, x_d, x_{d+1}^\ell)$) if $q = 2\ell$ (resp. $q = 2\ell + 1$). It is standard to check that $I$ is good. The assertion follows from Corollary \ref{2.6}.

(3) $A$ is a Noetherian local ring of dimension one, because $V = K[[t]]$ is the normalization of $A$ which is a module--finite extension of $A$ with $A: V = tV =\fkm$. Hence there are no proper intermediate subrings between $V$ and $A$. Let $I$ be an Ulrich ideal of $A$ and assume that $I$ is  not a parameter ideal. Choose $a \in I$ so that $Q= (a)$ is a reduction of $I$. Let $\frac{I}{a} = \{\frac{x}{a} \mid x \in I\}$. Then, because $\frac{I}{a} = A[\frac{I}{a}]$ and $\frac{I}{a} \ne A$, we get $I = aV$. Hence $I = A:V = tV = \fkm$, because $I = Q : I$.  
\end{proof}

\begin{rem} \label{2.8}
Let $A$ be the ring of Example \ref{2.7} (3) and  put $I = t^n\fkm$ with $0 < n \in \mathbb{Z}$. Then $I \cong \fkm$ as an $A$--module but $I$ is not Ulrich, since $I \ne \fkm$.
This simple fact shows that for given $\fkm$--primary ideals $I$ and $J$ of $A$, $I$ is not necessarily an Ulrich ideal of $A$, even though $I \cong J$ as an $A$--module and $J$ is an Ulrich ideal of $A$.
\end{rem}

\section{Ulrich modules}\label{ulmd}
In this section we shall explain the basic technique of induction which dates back to \cite{V}. Firstly, let us recall the definition of an Ulrich module.

\begin{defn}
Let $M$ be a finitely generated $A$-module. Then we say that $M$ is an Ulrich $A$--module with respect to $I$, if the following conditions are satisfied. 
\begin{itemize}\label{1.2}
\item[(1)] $M$ is a maximal Cohen--Macaulay $A$-module.
\item[(2)] $IM = QM$.
\item[(3)] $M/IM$ is $A/I$-free.
\end{itemize}
\end{defn}

If $d = 1$ and $I$ is an Ulrich ideal of $A$, then $I$ is an Ulrich $A$--module with respect to $I$. More generally, higher syzygy modules of Ulrich ideals are Ulrich $A$--modules, as the following theorem shows. For each $i \ge 0$, let $\operatorname{Syz}_A^i(A/I)$ denote the $i$--th syzygy module of $A/I$ in a minimal free resolution.

\begin{thm}\label{3.2}
Let $I$ be an Ulrich ideal of $A$ and suppose that $I$ is not a parameter ideal of $A$. Then for all $i\ge d$, $\operatorname{Syz}_A^i(A/I)$
is an Ulrich $A$--module with respect to $I$.
\end{thm}

Theorem \ref{3.2} is proven by induction on $d$. Let $I$ be an arbitrary $\fkm$--primary ideal of a Cohen--Macaulay local ring $A$ of dimension $d \ge 0$ and assume that $I$ contains a parameter ideal $Q = (a_1, a_2, \ldots, a_d)$ as a reduction. We now suppose that $d > 0$ and put $a = a_1$. Let
$$\overline{A} = A/(a), \ \ \overline{I} = I/(a), \ \ \text{and}\ \  \overline{Q} = Q/(a).$$
We then have the following.

\begin{lem}\label{3.3}
If $I$ is an Ulrich ideal of $A$, then $\overline{I}$ is an Ulrich ideal of $\overline{A}$.
\end{lem}

\begin{proof}
We have $\ol{I}^2 = \ol{Q}{\cdot}\ol{I}$, since $I^2 = QI$, while $\ol{I}/\ol{Q} = I/Q$ and $\ol{A}/\ol{I} = A/I$. Hence by Lemma \ref{2.3} $\ol{I}$ is an Ulrich ideal of $\ol{A}$, because $\ol{I}/\ol{Q}$ is $\ol{A}/\ol{I}$--free.
\end{proof}

\begin{lem}\label{3.4}
Suppose that $I/I^2$ is $A/I$-free.
Then 
$$ \Syz_A^i(A/I)/a \Syz_A^i(A/I) \cong \Syz_{\overline{A}}^{i-1}(A/I) \oplus \Syz_{\overline{A}}^{i}(A/I)$$
for all $i \geq 1$.
\end{lem}

\begin{proof}
We look at the minimal free resolution 
$$\mathbb{F}_{\bullet} : \cdots \to F_i \overset{\partial_i}{\to} F_{i-1} \to \cdots \to F_1 \overset{\partial_1}{\to} F_0 = A \overset{\varepsilon}{\to} A/I \to 0$$
of $A/I$. Then, since $a$ is $A$--regular and $a{\cdot}(A/I) = (0)$, we get exact sequences
$$(3.5)\ \ \ \cdots \to F_i/aF_i \to F_{i-1}/aF_{i-1} \to \cdots \to F_1/aF_1  \to I/aI \to 0$$ and $0 \to A/I \overset{\varphi}{\to} I/aI \to \ol{A} \to A/I \to 0$ of $\ol{A}$--modules, where $\varphi (1) = \ol{a}$, the image of $a$ in $I/aI$. We claim that $\varphi$ is a split monomorphism. Namely

\begin{claim*}[Vasconcelos \cite{V}]\label{3.6}
$I/aI \cong A/I \oplus \ol{I}.$
\end{claim*}

\begin{proof}[Proof of Claim]
Let $n = \mu_A(I)$ and write $I=(x_1, x_2, \ldots, x_n)$ with $x_1 = a$. 
Then $I/aI=A \ol{x_1}+\sum_{i=2}^nA\ol{x_i}$, where $\ol{x_i}$ denotes the image of $x_i$ in $I/aI$. To see that $A \ol{x_1}\cap \sum_{i=2}^nA\ol{x_i} = (0)$, let $c_1, c_2, \ldots, c_n \in A$ and assume that $c_1\ol{x_1} + \sum_{i=2}^nc_i\ol{x_i} = 0$. Then, since $c_1a + \sum_{i=2}^nc_ix_i \in aI$, we have 
$(c_1 - y)a + \sum_{i=2}^nc_ix_i = 0$ for some $y \in I$. Now remember that $I/I^2 \cong (A/I)^{n}$. Hence the images of $x_i$~($1 \le i \le n$) in $I/I^2$ form an $A/I$--free basis of $I/I^2$, which shows $c_1 - y \in I$. Thus $c_1 \in I$, so that $c_1\ol{x_1}= 0=\sum_{i=2}^nc_i\ol{x_i}$ in $I/aI$. Hence $\varphi$ is a split monomorphism and $I/aI \cong A/I \oplus \ol{I}$.
\end{proof}

By (3.5) and Claim we get $\Syz_A^i(A/I)/a \Syz_A^i(A/I) \cong \Syz_{\overline{A}}^{i-1}(A/I) \oplus \Syz_{\overline{A}}^{i}(A/I)$ for $i \ge 2$. See Claim for the isomorphism in the case where $i = 1$.
\end{proof}

We are now in a position to prove Theorem \ref{3.2}.

\begin{proof}[Proof of Theorem $\ref{3.2}$]
Suppose that $I$ is an Ulrich ideal of $A$ with $\mu_A(I) > d$. If $d = 0$, then $I^2 = (0)$ and $I$ is $A/I$--free, so that $\Syz_A^i(A/I) \cong (A/I)^{n^i}$ for all $i \ge 0$, which are certainly Ulrich $A$--modules with respect to $I$. Let $d > 0$ and assume that our assertion holds true for $d-1$. Let $\ol{A}=A/(a)$ and $\ol{I} = I/(a)$, where $a = a_1$. Then $\ol{I}$ is an Ulrich ideal of $\ol{A}$ by Lemma \ref{3.3} and for all $i \ge d$ one has $\Syz_A^i(A/I)/a \Syz_A^i(A/I) \cong \Syz_{\overline{A}}^{i-1}(A/I) \oplus \Syz_{\overline{A}}^{i}(A/I)$ by Lemma \ref{3.4}. Because the right hand side of the above isomorphism is an Ulrich $\ol{A}$--module with respect to $\ol{I}$, it readily follows that $\Syz_A^i(A/I)$ is an Ulrich $A$--module with respect to $I$.
\end{proof}

\begin{rem} Vasconcelos \cite{V} proved that if a given ideal $I$ in a Noetherian local ring  $A$ has finite projective dimension and if $I/I^2$ is $A/I$--free, then $I$ is generated by an $A$--regular sequence. In his argument the key observation is Claim in Proof of Lemma \ref{3.4}, which shows every Ulrich ideal of finite projective dimension is a parameter ideal. Hence, inside regular local rings, Ulrich ideals are exactly parameter ideals.
\end{rem}

\section{Relations between Ulrich ideals and Ulrich modules}\label{reluium}
Let $(A,\fkm)$ be a Cohen--Macaulay local ring of dimension $d \ge 0$ and let $I$ be an $\fkm$--primary ideal of $A$ which contains a parameter ideal $Q = (a_1, a_2, \ldots, a_d)$ as a reduction. With this notation the converse of Theorem \ref{3.2} is also true and we have the following.

\begin{thm}[cf. {\cite[Proposition (2.5)]{BHU}}]\label{4.1}
The following conditions are equivalent.
\begin{itemize}
\item[(1)] $I$ is an Ulrich ideal of $A$ with $\mu_A(I) > d$.
\item[(2)] For all $i \ge d$, $\Syz_A^i(A/I)$ is an Ulrich $A$-module with respect to $I$.
\item[(3)] There exists an Ulrich $A$--module $M$ with respect to $I$ whose first syzygy module $\Syz_A^1(M)$ is an Ulrich $A$--module with respect to $I$.
\end{itemize}
When $d>0$, then we can add the following condition.
\begin{itemize}
\item[(4)] $\mu_A(I) > d$, $I/I^2$ is $A/I$-free, and $\Syz_A^i(A/I)$ is an Ulrich $A$-module with respect to $I$ for some $i \geq d$.
\end{itemize}
\end{thm}

The proof of Theorem \ref{4.1} is based on the following.

\begin{lem}\label{4.2}
Suppose that $ 0 \to X \to F \to Y \to 0 $
is an exact sequence of finitely generated $A$--modules such that
$F$ is free, $X \subseteq \fkm F$, and $Y$ is an Ulrich $A$--module with respect to $I$. Then the following conditions are equivalent.
\begin{itemize}
\item[(1)] $X= \Syz_A^1(Y)$ is an Ulrich $A$--module with respect to $I$.
\item[(2)] $I$ is an Ulrich ideal of $A$ with $\mu_A(I) > d$.
\end{itemize}
\end{lem}

\begin{proof} Since $Y$ is a maximal Cohen--Macaulay $A$--module, $X$ is also a maximal Cohen--Macaulay $A$--module if $X \ne (0)$. Look at the exact sequence
$ 0 \to X/QX \to F/QF \to Y/QY \to 0 $ and we get
$X/QX \cong (I/Q)^r$
where $r = \operatorname{rank}_AF > 0$, because $Y/QY \cong (A/I)^r$ and $X \subseteq \fkm F$. Remember that this holds true for any parameter ideal $Q$ of $A$ which is contained in $I$ as a reduction.

$(2) \Rightarrow (1)$
Since  $I^2 \subseteq Q$, we get $I{\cdot}(X/QX) = (0)$, so that $IX = QX$. Because $I/Q \cong (A/I)^{n-d}$ ~($n = \mu_A(I) > d$) by Lemma \ref{2.3}, $X \ne (0)$ and $X/IX = X/QX$ is a free $A/I$--module. Hence $X$ is an Ulrich $A$-module with respect to $I$.

$(1) \Rightarrow (2)$
Enlarging the residue class field of $A$ if necessary, we may assume that the field $A/\fkm$ is infinite. Since $X/IX \cong (I/Q)^r$ and $X/IX$ is $A/I$--free, we have $I^2 \subseteq Q$ and $I/Q$ is $A/I$--free. Thus $I$ is an Ulrich ideal of $A$ by Proposition \ref{2.4}. Notice that $I \ne Q$, because $(I/Q)^r \cong X/IX \ne (0)$.
Hence $\mu_A(I) > d$.
\end{proof}

Let us prove Theorem \ref{4.1}.

\begin{proof}[Proof of Theorem $\ref{4.1}$]
The implication $(1) \Rightarrow (2)$ follows by Theorem \ref{3.2}, while $(2) \Rightarrow (3)$ is obvious.
As for the implications $(2) \Rightarrow (4)$ and $(3) \Rightarrow (1)$, see Lemma \ref{4.2}.

Suppose that $d > 0$ and we will prove $(4) \Rightarrow (1)$. Let $a = a_1$ and put $\ol{A} = A/(a)$, $\ol{I} = I/(a)$, and $\ol{Q} = Q/(a)$. Then $\Syz_A^i(A/I)/a\Syz_A^i(A/I) \cong \Syz_{\ol{A}}^{i-1}(A/I) \oplus \Syz_{\ol{A}}^{i}(A/I)$ by Lemma \ref{3.4} and $\Syz_{\ol{A}}^{i-1}(A/I) \oplus \Syz_{\ol{A}}^{i}(A/I)$ is an Ulrich $\ol{A}$--module with respect to $\ol{I}$, since $\Syz_A^i(A/I)$ is an Ulrich $A$--module with respect to $I$. Therefore $\Syz_{\ol{A}}^{i-1}(A/I) \ne (0)$. If $\Syz_{\ol{A}}^{i}(A/I) = (0)$, then $\ol{I}$ has finite projective dimension, so that $\ol{I} = \ol{Q}$ (see Theorem \ref{3.2} or \cite{V}), which is impossible because $I \ne Q$. Hence $\Syz_{\ol{A}}^{i}(A/I) \ne (0)$. Consequently, $\ol{I}$ is an Ulrich ideal of $\ol{A}$, thanks to the implication $(3) \Rightarrow (1)$, and  hence $I^2 \subseteq Q$ and $I/Q$ is $A/I$--free, because $\ol{I}^2 = \ol{Q}{\cdot}\ol{I}$ and $\ol{I}/\ol{Q}$ is $\ol{A}/\ol{I}$--free. It is now easy to check that $I^2 = QI$. In fact, write $I = (x_1, x_2, \ldots, x_n)$~($n = \mu_A(I)$) with $x_i = a_i $ for $1 \le i \le d$. Then, for each $y \in I^2$, writing $y = \sum_{i=1}^dc_ix_i$ with $c_i \in A$, we see that $\sum_{i=1}^dc_i\ol{x_i} = 0$ in $I/I
 ^2$ where $\ol{x_i}$ denotes the image of $x_i$ in $I/I^2$. Consequently, $c_i \in I$ for all $1 \le i \le d$, because $\{\ol{x_i}\}_{1 \le i \le n}$ forms an $A/I$--free basis of $I/I^2$. Hence $y \in QI$, so that $I^2 = QI$, which shows $I$ is an Ulrich ideal of $A$.
\end{proof}

\begin{rem}(\cite[Example (2.6)]{BHU})
Suppose $d = 0$ and let $\ell = \min \{\ell \in \mathbb{Z} \mid \fkm^{\ell} = (0) \}$. Then $\fkm^{\ell - 1}$ is an Ulrich $A$--module with respect to $\fkm$, but if $\ell > 2$, $\fkm$ itself is not an Ulrich ideal of $A$, since $\fkm^2 \ne (0)$. Therefore an $\fkm$--primary ideal $I$ is not necessarily an Ulrich ideal of $A$, even though there exists an Ulrich $A$--module with respect to $I$. More precisely, let $A = k[[t]]/(t^3)$ where $k[[t]]$ is the formal power series ring over a field $k$. We look at the exact sequence
$0 \to \fkm^2 \to A \overset{t}{\to} A \to A/\fkm \to 0.$
Then $\fkm^2 = \Syz_A^2(A/\fkm)$ is an Ulrich $A$--module with respect to $\fkm$, but $\fkm$ is not an Ulrich ideal of $A$. This shows the implication $(4) \Rightarrow (1)$ in Theorem \ref{4.1} does not hold true in general, unless $ d = \dim A > 0$.
\end{rem}

\section{Duality}\label{dual}
Let $\rmK_A$ be the canonical module of $A$ (\cite{HK}). In this section we study the question of when the dual $M^\vee = \Hom_A(M, \rmK_A)$ of an Ulrich $A$--module $M$ is Ulrich. Our answer is the following.

\begin{thm}\label{5.1} Let $M$ be an Ulrich $A$--module with respect to $I$. Then the following conditions are equivalent.
\begin{enumerate}
\item[$(1)$] $M^{\vee} = \Hom_A(M, \rmK_A)$ is an Ulrich $A$--module with respect to $I$.
\item[$(2)$] $A/I$ is a Gorenstein ring.
\end{enumerate}
\end{thm}

\begin{proof}
Notice  that $M^{\vee}$ is a maximal Cohen--Macaulay $A$--module (\cite[Satz 6.1]{HK}). Since $IM = QM$ and $M^{\vee}/QM^\vee \cong \Hom_{A/Q}(M/QM, \rmK_{A/Q})$
(\cite[Korollar 6.3]{HK}), we get $IM^\vee = Q M^\vee$, while
$$\Hom_{A/Q}(M/IM, \rmK_{A/Q})  \cong \Hom_{A/Q}(A/I, \rmK_{A/Q})^m \cong (\rmK_{A/I})^m $$
by \cite[Korollar 5.14]{HK}, because $M/IM \cong (A/I)^m$ where $m = \mu_A(M) > 0$. Hence
$M^\vee/IM^\vee = M^\vee/ Q M^\vee \cong (\rmK_{A/I})^m.$
Therefore $M^\vee$ is an Ulrich $A$--module with respect to $I$ if and only if $\rmK_{A/I}$ is a free $A/I$--module, that is $A/I$ is a Gorenstein ring.
\end{proof}

As an immediate consequence of Corollary \ref{2.6} and Theorem \ref{5.1}, we get the following, where $M^* =\Hom_A(M,A)$ for each $A$--module $M$.

\begin{cor}\label{5.2}
Suppose that $A$ is a Gorenstein ring and let $I$ be an Ulrich ideal of $A$. Let $M$ be a maximal Cohen--Macaulay $A$--module. Then the following are equivalent.
\begin{enumerate}
\item[$(1)$] $M^*$ is an Ulrich $A$--module with respect to $I$.
\item[$(2)$] $M$ is an Ulrich $A$--module with respect to $I$.
\end{enumerate}
\end{cor}

Suppose that $A$ is a Gorenstein ring and let $M$ be a maximal Cohen--Macaulay $A$--module with a minimal free resolution
$ \cdots \to F_i \to \cdots \to F_2 \overset{\partial_2}{\to} F_1 \overset{\partial_1}{\to} F_0 \to M \to 0.$
Let $\Syz_A^1(M) = \operatorname{Im}\partial_1$ and put $\operatorname{Tr}M = \operatorname{Coker}\partial_1^*$, the Auslander transpose  of $M$. Then we get the presentation $$0 \to M^* \to F_0^* \overset{\partial_1^*}{\to} F_1^* \to \operatorname{Tr}M \to 0$$ of $\operatorname{Tr}M$, so that $[\operatorname{Tr}M]^* =  \Syz_A^2(M)$. Because the dual sequence
$$0 \to M^* \to F_0^* \overset{\partial_1^*}{\to} F_1^* \overset{\partial_2^*}{\to} F_2^* \to \cdots \to F_i^* \to \cdots $$
is exact, $\operatorname{Tr}M$ is a maximal Cohen--Macaulay $A$--module, if $\operatorname{Tr}M \ne (0)$, that is the case where $M$ is not free. Notice that 
$$M^* = \Syz_A^2(\operatorname{Tr}M),$$ if $M$ contains no direct summand isomorphic to $A$.

With this notation, when $A$ is a Gorenstein ring, we can modify Lemma \ref{4.2} in the following way.

\begin{cor}\label{5.3}
Suppose that $A$ is Gorenstein and let $I$ be an Ulrich ideal of $A$ which is not a parameter ideal. Let $M$ be a maximal Cohen--Macaulay $A$--module and assume that $M$ contains no direct summand isomorphic to $A$. Then the following are equivalent.
\begin{enumerate} 
\item[$(1)$] $M$ is an Ulrich $A$--module with respect to $I$.
\item[$(2)$] $M^*$ is an Ulrich $A$--module with respect to $I$.
\item[$(3)$] $\Syz_A^1(M)$ is an Ulrich $A$--module with respect to $I$.
\item[$(4)$] $\operatorname{Tr}M$ is an Ulrich $A$--module with respect to $I$.
\end{enumerate}
\end{cor}

\begin{proof}
(1) $\Leftrightarrow$ (2) See Corollary \ref{5.2}.

(1) $\Rightarrow$ (3) See Lemma \ref{4.2}. 

(3) $\Rightarrow (1)$
Let $X = \Syz_A^1(M)$ and look at the presentation  $0 \to X \to F_0 \to M \to 0$ of $M$ such that $F_0$ is a finitely generated free $A$--module and $X \subseteq \fkm F_0$. Take the $A$--dual and we get the exact sequence $0 \to M^* \to F_0^* \to X^* \to 0.$ Then by Corollary \ref{5.2}, $X^*$ is an Ulrich $A$--module with respect to $I$. Therefore $\Syz_A^1(X^*)$ is by Lemma \ref{4.2} an Ulrich $A$--module with respect to $I$, if $\Syz_A^1(X^*) \ne (0)$. On the other hand, because $M^* \cong \Syz_A^1(X^*) \oplus A^r$ for some $r \ge 0$ and because the reflexive $A$--module $M$ contains no direct summand isomorphic to $A$, we have $M \cong [\Syz_A^1(X^*)]^*$. Hence  by Corollary \ref{5.2}, $M$ is an Ulrich $A$--module with respect to $I$.

(1) $\Rightarrow (4)$ 
Because $M$ is an Ulrich $A$--module with respect to $I$, $[\operatorname{Tr}M]^* =  \Syz_A^2(M)$ is by Lemma \ref{4.2}  an Ulrich $A$--module with respect to $I$. Hence $\operatorname{Tr}M$ is by Corollary \ref{5.2}  an Ulrich $A$--module with respect to $I$.

(4) $\Rightarrow (2)$ This follows from Lemma \ref{4.2}, since $M^* = \Syz_A^2(\operatorname{Tr}M)$.
\end{proof}

\section{Ulrich ideals of numerical semigroup rings}\label{uinsr}
It seems interesting to ask, in a given Cohen--Macaulay local ring $A$, how many Ulrich ideals are contained, except parameter ideals. If $A$ is regular, we have nothing (\cite{V}), but in general cases the research is still in progress and we have no definitive answer. Here let us note a few results in a rather special case, that is the case where $A$ is a numerical semigroup ring over a field.

Let $k$ be a field. Let $a_1,  a_2, \ldots, a_\ell > 0 $~($\ell \ge 1$) be integers with $\operatorname{GCD}(a_1, a_2, \ldots, a_\ell) = 1$. We put $$H = \langle a_1, a_2, \ldots, a_\ell \rangle = \left\{\sum_{i=1}^\ell c_ia_i\,\Bigg|\,0 \le c_i \in \mathbb{Z}\right\}$$ which is  the numerical semigroup generated by $a_i's$. Let  
$$A = k[[t^{a_1}, t^{a_2},\ldots, t^{a_\ell}]] \ \ \subseteq \ \ k[[t]],$$
where $V = k[[t]]$ is the formal power series ring over $k$. Then the numerical semigroup ring $A$ of $H$ is a one-dimensional complete local integral domain with $V$ the normalization. Let 
${\mathcal X}_A^g$ denote the set of  Ulrich ideals of $A$ which are not  parameter ideals of $A$ but generated by monomials in $t$. 
We then have the following.

\begin{thm}\label{6.1}
The set ${\mathcal X}_A^g$ is finite.
\end{thm}

\begin{proof}
Let $I \in {\mathcal X}_A^g$ and put $a = \min \{h \in H \mid t^h \in I \}$. Then $Q = (t^a)$ is a reduction of $I$, since $t^aV = IV$. Therefore $I^2 = t^aI$.
As $I/Q\subseteq IV/Q=t^aV/t^aA\cong V/A$, we have $\ell_A(I/Q) \le \ell_A(V/A)= \sharp (\mathbb{N} \setminus H),$
which yields $\fkm^q {\cdot} (I/Q) = (0)$ where $q = \sharp (\mathbb{N} \setminus H)$. Therefore $\fkm^q \subseteq I$, since $I/Q \cong (A/I)^{n-1}$ by Lemma \ref{2.3} where $n = \mu_A(I) > 1$. Thus the set ${\mathcal X}_A^g$ is finite, because the set $\{h \in H \mid t^h \not\in \fkm^q \}$ is finite.
\end{proof}

Let us examine the following example.

\begin{ex}\label{6.2} ${\mathcal X}_{k[[t^3, t^5, t^7]]}^g = \{\fkm \}$.
\end{ex}

\begin{proof} We put $A = k[[t^3, t^5, t^7]]$. 
As $\fkm^2 = t^3\fkm$, we get $\fkm \in {\mathcal X}_{A}^g$. Let $I \in  
{\mathcal X}_{A}^g$. Then $1 < \mu_A(I) \le 3 = \rme_\fkm^0(A)$ (\cite{S2}). Suppose that $\mu_A(I) = 2$ and write $I = (t^a, t^b)$ with $a, b \in H$, $a < b$. Then $Q = (t^a)$ is a reduction of $I$ and $I/Q = (\ol{t^b}) \cong A/I$, where $\ol{t^b}$ denotes the image of $t^b$ in $I/Q$. Hence $I =(t^h \mid h \in H, h + (b-a) \in H),$
as $(t^a) : t^b = I$. Therefore,
 since $H \ni c$ for all $c \ge 5$, we get $t^3, t^5, t^7 \in I$ if $b - a \ge 2$, so that $I ~= \fkm$. This is impossible, because $\mu_A(\fkm) = 3$. If $b -a =1$, then $I =(t^5, t^6, t^7)$, which is also impossible. Hence $\mu_A(I) \ne 2$.

Let $I = (t^a, t^b, t^c)$ with $a,b,c \in H$ such that $a < b < c$. We put $Q = (t^a)$. Then $I/Q = (\ol{t^b}, \ol{t^c}) \cong (A/I)^2.$ Hence $(t^a) : t^c = I$, so that $I = (t^h \mid h \in H, h + (c-a) \in H)$. Because $c-a \ge 2$, we see $t^3, t^5, t^7 \in I$, whence $I = \fkm$ as is claimed.
\end{proof}

 When $A$ is a Gorenstein ring, that is the case where the semigroup $H$ is symmetric, we have the following characterization of Ulrich ideals generated by monomials.

\begin{thm}\label{6.3}
Suppose that $A=k[[t^{a_1},t^{a_2}, \ldots, t^{a_\ell}]]$ is a Gorenstein ring and let 
$I$
be an ideal of $A$. Then the following conditions are equivalent.
\begin{enumerate}
\item[$(1)$] $I \in {\mathcal X}_A^g$.
\item[$(2)$] $I=(t^a, t^b)$ ($a,b\in H, a<b$) and if we put $c=b-a$, the following conditions hold. 
\begin{enumerate}
\item[$\mathrm{(i)}$] $c \not\in H$, $2c \in H$,
\item[$\mathrm{(ii)}$] the numerical semigroup $H_1 = H + \left< c \right>$ is symmetric, and
\item[$\mathrm{(iii)}$] $a = \min \{h \in H \mid h + c \in H\}$.
\end{enumerate}
\end{enumerate}
\end{thm}

\begin{proof}
$(1) \Rightarrow (2)$ 
We have $\mu_A(I) = 2$ (Corollary \ref{2.6}). Let us write $I = (t^a, t^b)$~($a, b \in H$, $a < b$) and put $Q = (t^a)$. 
Then $I^2 = QI$. 
Therefore $t^{2b} \in (t^{2a}, t^{a+b})$, whence $t^{2b} \in (t^{2a})$, because $t^b \not\in Q = (t^a)$. 
Thus $b-a \not\in H$ but $2(b-a) \in H$. 
We put  $c = b-a$ and let $$B = k[[t^{a_1}, t^{a_2} , \ldots, t^{a_{\ell}}, t^{c}]]$$ be the semigroup ring of $H_1 = H + \left<c \right>$. 
Then, since $2c \in H$, we see $B = A + At^c$, so that $t^aB = t^a A + t^b A = I$. Because $I/Q = t^aB/t^a A \cong B/A$ and $I/Q \cong A/I$, we have $I = A : B$. 
Hence $B$ is a Gorenstein ring, because $\rmK_B \cong A:B = I$ by \cite[Satz 5.22]{HK} and $I = t^aB$. 
Thus $H_1 = H + \left<c\right>$ is symmetric. 
Assertion (iii) is now clear, since $I = Q : I = (t^h \mid h \in S)$ and $IV = t^aV$ where $S=\{h \in H \mid h + c \in H\}$.

$(2) \Rightarrow (1)$
We put $Q = (t^a)$. 
Then $I^2 = QI$ and $I \ne Q$ by (i) and (ii). 
We must show $I = Q : I$. Let $B = A[t^c]$. 
Then, since $t^{2c} \in A$, we get $B = A + At^c$, so that $t^aB = I$. 
Hence $A:B = Q : I$. Let $J = A : B$. 
We then have $J = fB$, because $A : B \cong \rmK_B$ and $B = k[[t^{a_1}, t^{a_2} , \ldots, t^{a_{\ell}}, t^{c}]]$ is a Gorenstein ring by (ii). 
Hence $\frac{I}{t^a} = B = \frac{J}{f}.$
On the other hand, because $J=Q:I = (t^h \mid h \in S)$ where $S = \{h \in H \mid h + c \in H\}$, by (iii) $Q$ is a reduction of $J$ (notice that $t^a V = JV$), whence $J^2 = t^a J$ (remember that $J^2 = fJ$). 
Consequently
$\frac{J}{t^a} = \frac{J}{f} = \frac{I}{t^a},$
whence $I = J = Q : I$. 
Thus  $I/Q = (\ol{t^b}) \cong A/I$ where $\ol{t^b}$ is the image of $t^b$ in $I/Q$. Hence $I \in {\mathcal X}_A^g$ as claimed.
\end{proof}

\begin{cor}\label{6.5}
Let $a \ge 5$ be an integer. Then ${\mathcal X}_{k[[t^a, t^{a+1}, \ldots, t^{2a-2}]]}^g = \emptyset$.
\end{cor}

\begin{proof}
We put $H = \left<a, a+1, \ldots, 2a - 2 \right>$. Then $H$ is symmetric. 
Let $c \in \mathbb{Z}$. Assume that $c \not\in H$ but $2c \in H$ and put $H_1 = H + \left< c \right>$. Then $H_1 \setminus H = \{c, 2a-1\}$ and it is routine to check that $H_1$ is never symmetric, whence ${\mathcal X}_{k[[t^a, t^{a+1}, \ldots, t^{2a-2}]]}^g = \emptyset$ by Theorem \ref{6.3}.
\end{proof}

Using the characterization of Ulrich ideals of Theorem \ref{6.3}, we can determine all the Ulrich ideals 
of semigroups rings when $H$ is generated by $2$ elements.  For that purpose, we recall 
the following result of \cite{W}. 

\begin{lem}\label{Wat}\cite[Proposition 3]{W}
Let $H = \left<a, b, c \right>$ be a symmetric numerical 
semigroup generated minimally by $3$ integers.  Then changing the order of $a,b,c$ if necessary, we can write  $b = b'd, c= c'd$ where $d>1, \GCD(a,d)=1$ and $a\in \left< b', c' \right>$.    
\end{lem}

Next we determine the structure of $H_1= H + \left< c \right>$ in Theorem 
\ref{6.3} when $H=\left< a, b \right>$.

\begin{lem}\label{6.6} Let $H=\left< a, b \right>$ and $H_1= H + \left< c \right>$
 be symmetric numerical semigroups, where $a,b >1$ are relatively coprime integers
and $c$ is a positive integer satisfying $c\not\in H$ and $2c\in H$.     
Then after changing the order of $a,b$ if necessary,  one of the following cases occur.
\begin{enumerate}
\item[$(1)$]  $H=\left< 2, 2\ell +1 \right>$ and $c= 2m+1$ with $0\le m< \ell$,  
\item[$(2)$]  $a= 2c/d$, where $d= \GCD(b,c)$ is odd and $d \ge 1$. 
\item[$(3)$]  $a=2d$, where $d= \GCD(a,c)>1, c/d$ is odd and $1\le c/d < b$. 
 \end{enumerate}
\end{lem}
\begin{proof}  If $H=\left< 2, 2\ell +1 \right>$, then obviously the case (1) occurs. 
Henceforth we assume $2\not\in H$ and $c\ne 1$.

If $H_1= H + \left< c \right>$ is generated by $2$ elements and $H_1=
\left< b, c\right>$,  we may assume 
 \begin{enumerate}
\item[$\mathrm{(i)}$] $a = mb + nc$ and 
\item[$\mathrm{(ii)}$] $2c = pa + qb$
\end{enumerate}
for some non-negative integers $m,n,p,q$.  From $\mathrm{(i)}$ and $\mathrm{(ii)}$
we get 
$$ 2a = 2mb + 2nc = npa + (2m +  nq)b.$$
Hence we must have $0\le np\le 2$ and if $np=1$, $a\in \left< b\right>$,  a contradiction. 
If $np=0$, since $a,b$ are relatively coprime,  we must have $b=2$, contradicting 
our hypothesis $2\not\in H$.  If $np=2$, then 
$m=q=0$ and $a= nc$  and since $np=2$ and $c \not\in H$, we must have $a=2c$.

Now, we assume that  $H_1$ is minimally generated  by $3$ elements and $a,b>2$.
Then by Lemma \ref{Wat}, we may assume 
$$b = b'd,\ c= c'd,$$ 
where $d>1, \GCD(a,d)=\GCD(b',c')=1$ and $a\in \left< b', c' \right>$. Then we have 

\begin{enumerate}
\item[$\mathrm{(i)}$] $a = mb' + nc'$ and 
\item[$\mathrm{(ii)}$] $2c = pa + qb$.
\end{enumerate}
We can put $p = p'd$ and $2c' = p'a + qb'$. Note that $n\ne 0$ since 
$a,b$ are relatively coprime. 
We have the following equality.
$$2a =   2mb' + 2nc' = np'a + (2m + nq) b'.$$  
Again, we must have $0\le np'\le 2$ and if $np'=1$, $a\in \left< b'\right>$,  a contradiction. 
If $np'=2$, then $m=q=0$ and we must have $a=2c'$, $(a,b,c) = (2c', b'd, c'd)$, with $d$ 
odd and $d>1$. \par 
If $p'=0$, then we have  $2c = qb$, or $2c' =  qb'$  and we must have $b'=2$.   Now, let us 
 interchange $a$ and $b$. Then $a=2d$ and $c=c'd$. Since $H_1$ is symmetric, 
 $b>c'$ by Lemma \ref{Wat}.  This is our case (3).  
\end{proof}

\begin{thm}\label{H=(a,b)}
Suppose that $A=k[[t^a,t^b]]$ with $\GCD(a,b)=1$. 
Then either $a$ or $b$ is even. 
\par
If we put $a=2d$ and $b= 2\ell+1$, then 
\[
 \chi_A^g=\left\{
(t^{ia},t^{db})
\right\}_{1 \le i \le \ell}.
\] 
\end{thm}

\begin{proof}
Let $I=(t^{\alpha},t^{\beta}) \in \chi_A^g$ with $c=\beta-\alpha > 0$. 
We have shown in Lemma \ref{6.6} that $a$ or $b$ is even.  
\par
In what follows, we put $a=2d$ and $b=2\ell+1$. 
By Theorem \ref{6.3} and Lemma \ref{6.6}, 
the number $c $ is determined as one of the following cases.
\begin{enumerate}
\item[$\mathrm{(i)}$] $c= ad/2$, where 
$d$ is a proper divisor of $b$ (including $d=1$) and  
\item[$\mathrm{(ii)}$]  $c= ac' /2$
where $c'$ is an odd integer with $\GCD(a,c')=1, c'< b$ and 
$\alpha=\min \{h \in H \mid h + c \in H\}$. 
\end{enumerate}
But we can easily see that the case (i) is included in the case (ii).  
In particular, $\sharp \chi_A^g=\ell$. So it is enough to show that 
$(t^{ia},t^{db})$ is an Ulrich ideal for every $i=1,\ldots,\ell$. 
Indeed, since $A \cong k[[X,Y]]/(X^{2\ell+1}-Y^{2d})$, if we put 
$I_i=(t^{ia},t^{db})=(x^i,y^d)$ and $Q_i=(t^{ia})=(x^i)$, then 
we have $I_i^2=Q_iI_i$, $e_{I_i}^0(A)=2 \cdot \ell_A(A/I_i)=2id$ and 
$\mu_A(I_i)=2$ for every $i$. 
Hence $I_i$ is an Ulrich ideal by Corollary \ref{2.6}, as required.  
\end{proof}

\begin{ex}\label{6.8} The following assertions hold true.
\begin{enumerate}
\item[$(1)$] ${\mathcal X}_{k[[t^3, t^5]]}^g = \emptyset$.
\item[$(2)$] If $A=k[[t^8,t^{15}]]$, then $c= \beta -\alpha$
is one of the integers $4,12,20,28,36,44$ and $52$. Hence 
${\mathcal X}_{A}^g =  \{ ( t^{8i}, t^{60})\}_{1\le i\le7}$.
\item[$(3)$] ${\mathcal X}_{k[[t^4, t^6, t^{4\ell -1}]]}^g = \{(t^4, t^6), (t^{4\ell - 4}, t^{4 \ell - 1}),(t^{4(\ell - q) - 6}, t^{4\ell - 1}), (t^{4(\ell - q) - 8}, t^{4\ell - 1})\}_{0 \le q \le \ell - 3},$
where $\ell \ge 2$. 
\end{enumerate}
\end{ex}



\section{Structure of minimal free resolutions of Ulrich ideals}\label{minfree}
Let $(A, \fkm)$ be a Cohen--Macaulay local ring of dimension $d \ge 0$ and let $I$ be an $\fkm$--primary ideal of $A$ which contains a parameter ideal $Q=(a_1,a_2,\cdots,a_d)$ as a reduction. The purpose of this section is to explore the structure of minimal free resolutions of Ulrich ideals.

Throughout this section, we assume that $I$ is an Ulrich ideal of $A$. Let 
$${\mathbb{F}}_{\bullet}: \cdots \to F_i \overset{{\partial}_i}{\to} F_{i-1} \to \cdots \to F_1 \overset{{\partial}_1}{\to} F_0=A \overset{\varepsilon}{\to} A/I \to 0$$
be a minimal free resolution of the $A$--module $A/I$.
We put $\beta_i=\rank_AF_i = \beta_i^A(A/I)$, the $i$-th Betti number of $A/I$, and $n = \mu_A(I) = \beta_1 \ge d$.

We begin with  the following.

\begin{thm}\label{7.1}
One has $A/I \otimes_A \partial_i=0$ for all $i \geq 1$, and 
$$
\beta_i= \left\{
\begin{array}{ll}
(n-d)^{i-d}{\cdot}(n-d+1)^d & (d \le i),\\
\binom{d}{i}+(n-d){\cdot}\beta_{i-1} & (1 \leq i \leq d),\\
1 & (i=0).
\end{array}
\right.
$$
Hence $\beta_i=\binom{d}{i}+(n-d){\cdot}\beta_{i-1}$ for all $i \geq 1$.
\end{thm}

\begin{proof}
We proceed by induction on $d$. If $d = 0$, then $\Syz_A^i(A/I) \cong (A/I)^{n^i}$~($ i \ge 0$) and the assertions are clear. 
Assume that $d >0$ and that our assertions hold true for $d-1$.
Let $a=a_1$ and put $\ol{A}=A/(a)$, $\ol{I}=I/(a)$. Then by Lemma \ref{3.4}
$$\Syz_A^i(A/I)/a \Syz_A^i(A/I) \cong \Syz_{\ol{A}}^{i-1}(A/I) \oplus \Syz_{\ol{A}}^{i}(A/I)$$
for each $i \ge 1$. Hence we get $\beta_i = \ol{\beta}_{i-1} + \ol{\beta}_i$ for $i \ge 1$, where $\ol{\beta}_i = \beta_i^{\ol{A}}(A/I)$. We put $\ol{n}=n-1$ and $\ol{d}=d-1$. Then, thanks to the hypothesis of induction on $d$, we get for each $i \ge d$ that 
\begin{align*}
\beta_i&=\ol{\beta}_{i-1}+\ol{\beta_i}
=(\ol{n}-\ol{d})^{i-1-\ol{d}}{\cdot}(\ol{n}-\ol{d}+1)^{\ol{d}}+(\ol{n}-\ol{d})^{i-\ol{d}}{\cdot}(\ol{n}-\ol{d}+1)^{\ol{d}}\\
&=(n-d)^{i-d}{\cdot}(n-d+1)^{d-1}+(n-d)^{i-d+1}{\cdot}(n-d+1)^{d-1}\\
&=(n-d)^{i-d}{\cdot}(n-d+1)^{d-1}{\cdot}\left[1+(n-d)\right]
=(n-d)^{i-d}{\cdot}(n-d+1)^{d}.
\end{align*}

Let $1 \leq i \leq d$.
If $i = 1$, then $\beta_i=\beta_1=n=\binom{d}{1}+(n-d){\cdot}\beta_0=\binom{d}{i}+(n-d){\cdot}\beta_{i-1}$. If $2 \leq i \leq d-1$, then
\begin{align*}
\beta_i&=\ol{\beta}_{i-1}+\ol{\beta}_i
=\textstyle\left[\binom{\ol{d}}{i-1}+(\ol{n}-\ol{d}){\cdot}\ol{\beta}_{i-2}\right]+\left[\binom{\ol{d}}{i}+(\ol{n}-\ol{d}){\cdot}\ol{\beta}_{i-1}\right]\\
&=\textstyle\left[\binom{d-1}{i-1}+(n-d){\cdot}\ol{\beta}_{i-2}\right]+\left[\binom{d-1}{i}+(n-d){\cdot}\ol{\beta}_{i-1}\right]\\
&=\textstyle\binom{d}{i}+(n-d){\cdot}\left[\ol{\beta}_{i-2}+\ol{\beta}_{i-1}\right]
=\textstyle\binom{d}{i}+(n-d){\cdot}\beta_{i-1}.
\end{align*}

Suppose that $i=d \geq 2$. We then have
$$
\beta_i=\beta_d= \ol{\beta}_{d-1} + \ol{\beta}_d= \textstyle\binom{\ol{d}}{\ol{d}} + (\ol{n}- \ol{d}){\cdot}\ol{\beta}_{d-2}+ \ol{\beta}_d= \textstyle\binom{d}{d} + (n -d){\cdot}\ol{\beta}_{d-2} + \ol{\beta}_d,
$$
while
\begin{align*}
\textstyle\binom{d}{d} + (n-d){\cdot}\beta_{d-1}
&= \textstyle\binom{d}{d} + (n-d){\cdot}\left[\ol{\beta}_{d-2} + \ol{\beta}_{d-1}\right]
= \textstyle\binom{d}{d} + (n-d){\cdot}\ol{\beta}_{d-2} + (n-d){\cdot}\ol{\beta}_{d-1}\\
&= \textstyle\binom{d}{d} + (n-d){\cdot}\ol{\beta}_{d-2} + (\ol{n} - \ol{d}){\cdot}\ol{\beta}_{d-1}
= \textstyle\binom{d}{d} + (n-d){\cdot}\ol{\beta}_{d-2} + \ol{\beta}_d.
\end{align*}
Hence
$$
\beta_i= \left\{
\begin{array}{ll}
(n-d)^{i-d}{\cdot}(n-d+1)^d & (d \le i),\\
\binom{d}{i}+(n-d){\cdot}\beta_{i-1} & (1 \leq i \leq d),\\
1 & (i=0),
\end{array}
\right.$$
so that $\beta_i=\binom{d}{i}+(n-d){\cdot}\beta_{i-1}$ for all $i \geq 1$.

Because 
$$  \cdots \to F_i/aF_i \to F_{i-1}/aF_{i-1} \to \cdots \to F_1/aF_1 \to I/aI \to 0$$
is a minimal free resolution of the $\ol{A}$--module $I/aI$ and because 
$I/aI \cong A/I \oplus \ol{I}$ by Claim in the proof of Lemma \ref{3.4}, we see $A/I \otimes_A\partial_i=\ol A/\ol I\otimes_{\ol A}\ol\partial_i= 0$ for $i > 1$ from the induction hypothesis, where $\ol\partial_i:=\ol A\otimes_A\partial_i$. As $A/I\otimes_A\partial_1 = 0$, this proves Theorem \ref{7.1}.
\end{proof}

Suppose that $d > 0$ and we look at the exact sequence
$$(\sharp) \ \ \ 0 \to Q \overset{\iota}{\to} I \to I/Q \to 0$$
of $A$--modules, where $\iota : Q \to I$ is the embedding. Remember now that a minimal free resolution of $Q$ is given by the truncation 
$$\mathbb{L}_{\bullet}: 0 \to K_d \to \cdots \to K_1 \to Q \to 0$$
of the Koszul complex ${\mathbb{K}}_{\bullet}  = \mathbb{K}_\bullet (a_1, a_2, \ldots, a_d; A)$ generated by the $A$--regular sequence $a_1, a_2, \ldots, a_d$ and a minimal free resolution of $I/Q = (A/I)^{n-d}$ is given by the direct sum $\mathbb {G}_{\bullet}$ of $n - d$ copies of $\mathbb{F}_{\bullet}$. Then by the horseshoe lemma, a free resolution of $I$ is induced from $\mathbb{L}_{\bullet}$ and $\mathbb {G}_{\bullet}$ via exact sequence $(\sharp)$ above. With this notation, what Theorem \ref{7.1} says is the following.

\begin{cor}\label{7.2} In the exact sequence $0 \to Q \overset{\iota}{\to} I \to (A/I)^{n-d} \to 0$, the free resolution of $I$ induced from $\mathbb{L}_{\bullet}$ and $\mathbb{G}_{\bullet}$ is a minimal free resolution.
\end{cor}

For example, suppose that $A$ is a Gorenstein ring with $\dim A = 0$ and assume that $I \ne (0)$. Then $I = (x)$ for some $x \in A$ (Lemma \ref{2.6}). Because $(0) : I = I$, a minimal free resolution of $A/I$ is given by 
$$\mathbb{F}_{\bullet} \cdots \to A \overset{x}{\to} A \overset{x}{\to} A  \to A/I \to 0.$$

We similarly have the following.

\begin{ex}\label{7.3}
Suppose that $A$ is a Gorenstein ring with $\dim A = 1$. Let $I$ be an $\fkm$--primary ideal of $A$ containing $Q = (a)$ as a reduction. Assume that $I$ is an Ulrich ideal of $A$ which is not a parameter ideal. Then $\mu_A(I) = 2$. We write $I = (a,x)$~($x \in A$). Then $x^2 = ay$ for some $y \in I$, because $I^2 = aI$. With this notation, a minimal free resolution of $A/I$ is given by

\vspace{2mm}

$$\mathbb{F}_{\bullet} : \cdots \to A^2 \mapright{\begin{pmatrix}
-x&-y\\
a & x
\end{pmatrix}}
A^2 \mapright{\begin{pmatrix}
-x&-y\\
a & x
\end{pmatrix}}A^2 \mapright{\begin{pmatrix}
a&x
\end{pmatrix}} A \overset{\varepsilon}{\to} A/I \to 0.$$ 
\end{ex}

\begin{proof}
It is standard to check that $\mathbb{F}_{\bullet}$ is a complex of $A$--modules. To show that $\mathbb{F}_{\bullet}$ is exact, let $f,g \in A$ and assume that $af + xg = 0$. Then, since $g \in Q : I =I$, we may write $g = ag_1  + xg_2$ with $g_i \in A$. Then, because 
$af + xg = af + a(xg_1 + yg_2) = 0$, we get $f = -(xg_1+yg_2)$, so that 
$\begin{pmatrix}
f\\
g
\end{pmatrix} = \begin{pmatrix} -(xg_1 + yg_2)\\
ag_1 + xg_2
\end{pmatrix}
= \begin{pmatrix}
-x&-y\\
a & x
\end{pmatrix}
\begin{pmatrix}
g_1 \\
g_2
\end{pmatrix}.$
Therefore, if $f,g \in A$ such that 
$\begin{pmatrix}
-x&-y\\
a&x
\end{pmatrix}\begin{pmatrix}f\\
g
\end{pmatrix}= \binom{0}{0},$ we then have $
\begin{pmatrix}
f\\
g
\end{pmatrix} = \begin{pmatrix}
-x&-y\\
a & x
\end{pmatrix}
\begin{pmatrix}
g_1\\
g_2
\end{pmatrix}$ for some $g_i \in A$, because $af + xg = 0$. Hence $\mathbb{F}_\bullet$ is a minimal free resolution of $A/I$.
\end{proof}

As we have seen in Example \ref{7.3}, minimal free resolutions of Ulrich ideals of a Gorenstein local ring are eventually periodic.
Namely we have the following.

\begin{cor}\label{7.4}
The following assertions hold true.
\begin{itemize}
\item[(1)] $\Syz_A^{i+1}(A/I) \cong [\Syz_A^{i}(A/I)]^{n-d}$ for all $i \geq d$.
\item[(2)] Suppose that $A$ is a Gorenstein ring. Then one can choose a minimal free resolution ${\mathbb{F}}_{\bullet}$ of $A/I$ of the form 
$$\cdots \to F_d \overset{{\partial}_{d+1}}{\to} F_{d} \overset{\partial_{d+1}}{\to} F_d \overset{{\partial}_{d}}{\to} F_{d-1} \to \cdots \to F_1 \overset{\partial_1}{\to} F_0 = A \overset{\varepsilon}{\to} A/I \to 0,$$ that is $F_{d+i}=F_d$ and $\partial_{d+i+1}=\partial_{d+1}$ for all $i \geq 1$.
\end{itemize}
\end{cor}

\begin{proof}
(1) This is clear, because
$$\Syz_A^{i+1}(A/I) = \Coker~\partial_{i+2} \cong \left[\Coker~\partial_{i+1}\right]^{n-d} = \left[\Syz_A^i(A/I)\right]^{n-d}$$
for all $i \ge d$ (see Corollary \ref{7.2}).

(2) By Example \ref{7.3} we may assume $n > d \ge 2$; hence $n = d+1$. By Corollary \ref{7.2} there exist isomorphisms $\alpha : F_{d+2} \tilde{\to} F_{d+1}$ and $\beta : F_{d+1} \tilde{\to} F_d$ which make the following diagram
\[\begin{array}{ccccccc}
F_{d+2} & \mapright{\partial_{d+2}} & F_{d+1} & \mapright{\partial_{d+1}} & F_{d} & \mapright{\partial_d} & F_{d-1}\\
\mapdown{\alpha} & &\mapdown{\beta} & & \mapdown{} & & \mapdown{}\\
F_{d+1} & \mapright{\partial_{d+1}} & F_d & \mapright{\partial_d} & F_{d-1} & \mapright{\partial_{d-1}} & F_{d-2}
\end{array}\]
commutative. Then a simple diagram chase will show that the sequence 
$$\cdots \to F_{d+1} \overset{\beta^{-1}\partial_{d+1}}{\to} F_{d+1} \overset{\beta^{-1}\partial_{d+1}}{\to} F_{d+1} \overset{\partial_d \beta}{\to}  F_{d-1} \overset{\partial_{d-1}}{\to} F_{d-2}$$
is exact.
\end{proof}

The second assertion of the above corollary yields the following result.

\begin{cor} \label{Hyp}
Let $(A,\m)$ be a Gorenstein local ring.
Suppose that there exist non-parameter Ulrich ideals $I,J$ of $A$ with $\m J\subseteq I\subsetneq J$.
Then $A$ is a hypersurface.
\end{cor}

\begin{proof}
The natural exact sequence $0 \to J/I \to A/I \to A/J \to 0$ induces an inequality $\beta_i^A(J/I)\le\beta_i^A(A/I)+\beta_{i+1}^A(A/J)$ for all integers $i$.
Using Theorem \ref{7.1}, we have $\beta_i^A(J/I)\le2^d+2^d=2^{d+1}$ for all $i\ge d$.
By assumption, $J/I\cong(A/\m)^n$ for some $n\ge1$.
Hence $\beta_i^A(A/\m)\le\frac{2^{d+1}}{n}$, which says that the Betti numbers of $A/\m$ are bounded above.
It follows from \cite[Remarks 8.1.1(3)]{A} that $A$ is a hypersurface.
\end{proof}

Assertion (1) of Corollary \ref{7.4} shows that Ulrich modules with respect to $I$ obtained by syzygies $\Syz_A^i(A/I)$ ~ ($i \ge d$) are essentially of one kind. To see this phenomenon more precisely, let $\rmI_1(\partial_i)$ ~($i \ge 1$) denote the ideal of $A$ generated by the entries of the matrix $\partial_i : F_i \to F_{i-1}$. We then have the following.

\begin{thm}\label{7.5} Suppose that $\mu_A(I) > d$. Then 
$\rmI_1(\partial_i)=I$ for all $i \geq 1$.
\end{thm}

\begin{proof}
The assertion is obvious, if $d = 0$ (remember that $\Syz_A^i(A/I) \cong (A/I)^{n^i}$ for all $i \ge 0$). Therefore induction on $d$ easily shows that $\rmI_1(\partial_i) + Q = I$ for all $i \ge 1$ (use Lemma \ref{3.4}).

Suppose now that $d > 0$. Then we get $\rmI_1(\partial_i) \supseteq Q$ for all $1 \le i \le d$, because by Corollary \ref{7.2} the truncation
$$\mathbb{L}_{\bullet} : 0 \to K_d \to K_{d-1} \to \cdots \to K_1 \to Q \to 0$$
of the Koszul complex $\mathbb{K}_{\bullet} = \mathbb{K}_{\bullet}(a_1, a_2, \ldots, a_d;A)$ is a subcomplex of the truncation 
$${\mathbb{M}}_{\bullet} \cdots \to F_{d+1} \to F_d \to \cdots \to F_1 \to I \to 0$$
of the minimal free resolution $\mathbb{F}_{\bullet}$ of $A/I$ and $K_i$ is a direct summand of $F_i$ for each $1 \le i \le d$. Hence $\rmI_1(\partial_i) = I$ if $1 \le i \le d$. On the other hand, Corollary \ref{7.2} shows also that $\rmI_1(\partial_{i+1}) = \rmI_1(\partial_i)$ for $i \ge d+1$. Consequently, to see that $\rmI_1(\partial_i) = I$ for all $i \ge d+1$, it suffices to show $\rmI_1(\partial_{d+1}) \supseteq Q$ only, which is obviously true, because by Corollary \ref{7.2} the matrix $\partial_{d+1}$ has the form
$$\partial_{d+1} = \begin{pmatrix}
*\\
\partial_d^{\oplus n-d}
\end{pmatrix}\ \ \ (n -d > 0)
$$
with $\rmI_1(\partial_d) = I$. This completes the proof of Theorem \ref{7.5}.
\end{proof}

The following result is a direct consequence of Corollary \ref{7.4} and Theorem \ref{7.5}.

\begin{cor}\label{7.6}
Let $(A,\fkm)$ be a Cohen--Macaulay local ring of dimension $d \ge 0$. Let $I$ and $J$ be $\fkm$--primary ideals of $A$ containing some parameter ideals of $A$ as reductions. Suppose that both $I$ and $J$ are Ulrich ideals of $A$ with $\mu_A(I) > d$ and $\mu_A(J) > d$. If $\Syz_A^i(A/I) \cong \Syz_A^i(A/J)$ for some $i\ge0$, then $I = J$.
\end{cor}

For a given Cohen--Macaulay local ring $A$ let
${\mathcal X}_A$ denote the set of Ulrich ideals $I$ of $A$ which contains parameter ideals as reductions but $\mu_A(I) > d$. Then as a consequence of Corollary \ref{7.6}, we get the following. Remember that $A$ is said to be of finite CM-representation type, if there exist only finitely many isomorphism classes of indecomposable maximal Cohen--Macaulay $A$--modules.

\begin{thm}\label{7.7}
If $A$ is of  finite CM--representation type, then the set $\calX_A$ is finite.
\end{thm}

\begin{proof}
Let $I \in \calX_A$.
Then $\mu_A(I) \le \rmr (A) + d$ (Corollary \ref{2.6}). Let $$\calS = \{\left[\Syz_A^d(A/I) \right] \mid I \in \calX_A \},$$
where $\left[\Syz_A^d(A/I) \right]$ denotes the isomorphism class of the maximal Cohen--Macaulay $A$--module $\Syz_A^d(A/I)$. Then because $\beta_d^{A}(A/I) = (\mu_A(I) - d + 1)^d \le \left[\rmr (A) +1\right]^d,$
the minimal number $\mu_A(\Syz_A^d(A/I))$ of generators for  $\Syz_A^d(A/I)$ has an upper bound which is independent of the choice of $I \in \calX_A$. Hence the set $\calS$ is finite, because $A$ is of  finite CM-representation type. Thus the set $\calX_A$ is also finite, because $\calX_A$ is a subset of $\calS$ by Corollary \ref{7.6}.
\end{proof}

Let us explore the following example in order to illustrate Theorem \ref{7.7}.

\begin{ex}\label{7.8}
Let $A = k[[X, Y, Z]]/(Z^2 - XY)$ where $k[[X,Y,Z]]$ is the formal power series ring over a field $k$. Then $\calX_A = \{\fkm \}$.
\end{ex}

\begin{proof}
Let $x, y$, and $z$ be the images of $X, Y$, and $Z$ in $A$, respectively. Then $\fkm^2 = (x,y)\fkm$, so that $\fkm \in \calX_A$. Let $I \in \calX_A$ and put $M = \Syz_A^2(A/I)$. Then $\mu_A(I) = 3$ (Corollary \ref{2.6}), $\rank_AM = 2$, and $\mu_A(M)=4$ (Theorem \ref{7.1}). Therefore, because $A$ and $\fkp = (x,z)$ are the indecomposable maximal Cohen--Macaulay $A$--modules (up to isomorphism), we get $M \cong \fkp \oplus \fkp \cong \Syz_A^2(A/\fkm),$ so that $I = \fkm$ by Corollary \ref{7.6}. Thus $\calX_A = \{\fkm \}$ as claimed.
\end{proof}

\section{Relations of Ulrich modules to linear resolutions}\label{linear}

In this section we discuss the relation of Ulrich modules to linear free resolutions.
We fix some notation all over again as follows.
Let $A$ be a commutative ring with nonzero identity.
Let $I$ be an ideal of $A$.
Put
$$R=\rmR(I)=A[It], \ \  R'=\rmR'(I)=A[It,t^{-1}] \subseteq A[t,t^{-1}], \ \ \mbox{and} \ \ \gr_I(A)=R'/t^{-1}R'$$
where $t$ stands for an indeterminate over the ring $A$.

\begin{defn}\label{9.1}
For an $A$-module $M$, ${\calF}=\{ M_n \}_{n \in \Z}$ is an $I$-filtration of $M$ if
\begin{itemize}
\item[(1)] $M_n$ is a submodule of $M$ for all $n \in \Z$,
\item[(2)] $M=M_0$, and
\item[(3)] $M_n \supseteq M_{n+1} \supseteq IM_n$ for all $n \in \Z$.
\end{itemize}
When this is the case, we say that the $A$-module $M$ is $I$-filtered with respect to $\calF$.
\end{defn}

We put
\begin{align*}
\rmR'(\calF)&=\sum_{n \in \Z}\{ t^n \otimes x \ | \ x \in M_n \} \subseteq A[t,t^{-1}] \otimes_A M,\\
\rmR(\calF)&=\sum_{n \geq 0}\{ t^n \otimes x \ | \ x \in M_n \} \subseteq A[t] \otimes_A M,\\
\gr(\calF)&=\rmR'(\calF)/t^{-1}\rmR'(\calF)
\end{align*}
for an $I$-filtration $\calF=\{ M_n \}_{n \in \Z}$ of $M$.

Let $L$ and $M$ are $A$-modules and $f:L \to M$ be an $A$-linear map.
Suppose that $\{L_n\}_{n \in \Z}$ and $\{M_n\}_{n \in \Z}$ are $I$-filtrations of $L$ and $M$, respectively, and assume that $f(L_n) \subseteq M_n$ for all $n \in \Z$.
Then we have linear maps of graded modules $ \rmR'(f): \rmR'(\{L_n\}_{n \in \Z}) \to \rmR'(\{M_n\}_{n \in \Z})$, $ \rmR(f): \rmR(\{L_n\}_{n \geq 0}) \to \rmR(\{M_n\}_{n \geq 0})$ and $\gr(f): \gr(\{L_n\}_{n \in \Z}) \to \gr(\{M_n\}_{n \in \Z})$, which are induced by the linear map $A[t,t^{-1}] \otimes_A f: A[t,t^{-1}] \otimes L \to A[t,t^{-1}] \otimes M$.

Let us note the following lemma.

\begin{lem}\label{9.2}
Let $L$, $M$, and $N$ be $A$-modules.
Let $\{L_n\}_{n \in \Z}$, $\{M_n\}_{n \in \Z}$, and $\{N_n\}_{n \in \Z}$ be $I$-filtrations of $L$, $M$, and $N$, respectively.
Then the following hold.
\begin{itemize}
\item[(1)] Let $f: L \to M$ and $g : M \to N$ be $A$-linear maps and assume that $f(L_n) \subseteq M_n$ and $g(M_n) \subseteq N_n$ for $n \in \Z$.
Then we have $\gr(g)\circ \gr(f)=\gr(g \circ f)$.
\item[$(2)$] We have $\gr(1_M)=1_{\gr(\{M_n \}_{n \in \Z})}$, where $1_M$ and $1_{\gr(\{M_n \}_{n \in \Z})}$ denote the identity maps of $M$ and $\gr(\{M_n \}_{n \in \Z})$ respectively. 
\item[(3)] Put $[L \oplus M]_n=L_n \oplus M_n$ for $n \in \Z$.
Then $\{ [L \oplus M]_n \}_{n \in \Z}$ is an $I$-filtration of $L \oplus M$ and we have 
$$\gr(\{ [L \oplus M]_n \}_{n \in \Z}) \cong \gr(\{ L_n\}_{n \in \Z}) \oplus \gr(\{M_n \}_{n \in \Z})$$
as graded $\gr_I(A)$-modules.
\item[(4)] Let $X$ be a submodule of $M$. Then $\{ M_n \cap X \}_{n \in Z}$ forms an $I$-filtration of $X$.
\end{itemize}
\end{lem}

\begin{defn}\label{9.3}
Let $f: L \to M$ be an $A$-linear map of $A$-modules $L$ and $M$, and assume that the $A$-modules $L$ and $M$ are $I$-filtered with respect to $\{L_n\}_{n \in \Z}$ and $\{M_n\}_{n \in \Z}$ respectively.
Then the map $f: L \to M$ is strict, if we have $f(L_n) = f(L) \cap M_n$ for all $n \in \Z$.
\end{defn}

Concerning a strict $A$-linear map of $I$-filtered modules, we note the following lemma.

\begin{lem}\label{9.4}
Let $L \overset{f}{\to} M \overset{g}{\to} N$ be an exact sequence of $A$-modules and assume that $L$, $M$, and $N$ are $I$-filtered $A$-modules with respect to $\{L_n\}_{n \in \Z}$, $\{M_n\}_{n \in \Z}$, and $\{N_n\}_{n \in \Z}$, respectively.
Then 
$$\gr(\{L_n\}_{n \in \Z}) \overset{\gr(f)}{\to} \gr(\{M_n\}_{n \in \Z}) \overset{\gr(g)}{\to} \gr(\{N_n\}_{n \in \Z})$$ 
forms an exact sequence of graded $\gr_I(A)$-modules, if the $A$-linear maps $f$ and $g$ are strict.
\end{lem}

From now on, let us assume that $A$ is a $d$-dimensional Noetherian local ring with maximal ideal $\m$.
We put $J=\gr_I(A)_+$ and $\calM=\m{\cdot}\gr_I(A)+\gr_I(A)_+$ denotes the unique graded maximal ideal of $\gr_I(A)$.
We denote by $L(\alpha)$, for each $\alpha \in \Z$, the graded module whose grading is given by $[L(\alpha)]_{n}=L_{\alpha+n}$ for all $n \in \Z$.
For a finitely generated $A$-module $M$ and an ideal $I$ of $A$, we put ${\rmR}'_I(M)={\rmR}'(\{I^nM\}_{n \in \Z})$.
Let $\gr_I(M)=\gr(\{I^nM\}_{n \in \Z})$ denote the associated graded module  of $M$ with respect to $I$.

The goal of this section is the following.

\begin{thm}\label{9.5}
Suppose that $A$ is a Cohen--Macaulay local ring and that $I$ is an Ulrich ideal of $A$.
Let $Q$ be a parameter ideal of $A$ which forms a reduction of $I$.
Assume that $M$ is an Ulrich $A$-module with respect to $I$.
Let
$$ \mathbb{F}_{\bullet}: \cdots \to F_i \overset{\partial_i}{\to} F_{i-1} \to \cdots \to F_1 \overset{\partial_1}{\to} F_0 \overset{\varepsilon}{\to} M \to 0 $$
be a minimal free resolution of the $A$-module $M$.
Suppose that $F_i$ is an $I$-filtered $A$-module with respect to ${\calF}_i=\{ I^{n-i}F_i \}_{n \in \Z}$ for $i \geq 0$ and $M$ is an $I$-filtered $A$-module with respect to an $I$-adic filtration $\{I^n M\}_{n \in \Z}$.
Then we have the following.
\begin{itemize}
\item[(1)] The $A$-linear maps $\varepsilon$ and $\partial_i$ are strict for all $i \geq 1$,
\item[(2)] the sequence
$$ \cdots \to \gr({\calF_i}) \overset{\gr(\partial_i)}{\to} \gr({\calF_{i-1}}) \to \cdots \to \gr({\calF_1}) \overset{\gr(\partial_1)}{\to} \gr({\calF_0}) \overset{\gr(\varepsilon)}{\to} \gr_I(M) \to 0 $$
forms a minimal free resolution of the graded $\gr_I(A)$-module $\gr_I(M)$,
\item[(3)] $\rank_AF_i=\mu_A(M)\cdot\{\mu_A(I)-d\}^i$ for all $i \geq 0$, and
\item[(4)] ${\rmI}_1(\partial_i)+Q=I$ for all $i \geq 1$.
\end{itemize}
\end{thm}

Therefore the associated graded module $\gr_I(M)$ of an Ulrich module $M$ with respect to an Ulrich ideal $I$ has a minimal free resolution
$$ \cdots \to \bigoplus^{r_i}\gr_I(A)(-i) \to \cdots \to \bigoplus^{r_1}\gr_I(A)(-1) \to \bigoplus^{r_0}\gr_I(A) \to \gr_I(M) \to 0 $$
of graded $\gr_I(A)$-modules, where $r_i$ denotes the $i$-th Betti number $\beta_i^A(M)$ of $M$ for $i \geq 0$.

Before giving a proof of Theorem \ref{9.5}, let us begin with the following.

\begin{lem}\label{9.6}
Suppose that $M$ is a finitely generated $I$-filtered $A$-module with respect to $\calF=\{ M_n \}_{n \in \Z}$.
Assume that the $I$-filtration $\calF=\{ M_n \}_{n \in \Z}$ of $M$ is stable, that is an $I$-filtration with $M_{n+1}=IM_n$ for all $n \gg 0$.
If
$\gr(\calF)=\gr_I(A) \cdot [\gr(\calF)]_q$
for some $q \in \Z$ then we have $M_n=I^{n-q}M$ for all $n \in \Z$.
\end{lem}

\begin{proof}
We have $\gr(\calF)=\sum_{i=1}^{\ell}\gr_I(A){\cdot}\xi_i$ for some $\ell >0$ and $\xi_i \in [\gr(\calF)]_q$ with $1 \leq i \leq \ell$.
Write $\xi_i=\overline{t^q \otimes x_i}$ with $x_i \in M_q$ for $1 \leq i \leq \ell$, where $\overline{t^q \otimes x_i}$ denotes the image of ${t^q \otimes x_i} \in \rmR'(\calF)$ in $\gr(\calF)$.
Let $n \in \Z$ and take $x \in M_n$, then we have
$$\overline{t^n \otimes x}=\sum_{i=1}^{\ell}{c_it^{n-q}}\xi_i =\sum_{i=1}^{\ell}{c_it^{n-q}} \{\overline{t^q \otimes x_i}\} \\
= \sum_{i=1}^{\ell} \overline{t^n \otimes c_ix_i}
=\overline{t^n \otimes \Sigma_{i=1}^{\ell} c_ix_i} \in \gr(\calF),$$
for some $c_i \in I^{n-q}$ with $1 \leq i \leq \ell$ where $\overline{t^n \otimes c_ix_i}$, $\overline{t^n \otimes \Sigma_{i=1}^{\ell}c_ix_i}$ denote the images of ${t^n \otimes c_ix_i}, \ {t^n \otimes \Sigma_{i=1}^{\ell}c_ix_i} \in \rmR'(\calF)$ in $\gr(\calF)$, respectively.
Therefore we have $x -\sum_{i=1}^{\ell}c_ix_i \in M_{n+1}$.
Thus we get $M_n \subseteq I^{n-q}M_q+M_{n+1}$ for all $n \in \Z$.

We furthermore have the following claim.

\begin{claim}\label{claim}
$M_n \subseteq I^{n-q}M_q+M_m$ for all $n$, $m \in \Z$.
\end{claim}

\begin{proof}[Proof of Claim \ref{claim}]
We proceed by induction on $m$.
It is clear that the case where $m \leq n$.
Assume that $m \geq n+1$ and that our assertion holds true for $m-1$.
Then, by the hypothesis of induction on $m$, we have $M_n \subseteq I^{n-q}M_q+M_{m-1}$.
We also have $M_{m-1} \subseteq I^{m-1-q}M_q +M_m$ by the above argument.
Hence, since $m-1-q \geq n-q$, we have
$$ M_n \subseteq I^{n-q}M_q+M_{m-1} \subseteq I^{n-q}M_q +\{I^{m-1-q}M_q +M_m\} \subseteq I^{n-q}M_q+M_m$$
as required.
\end{proof}

By Claim \ref{claim} we have
$$M_n \subseteq I^{n-q}M_q+M_m =I^{n-q}M_q,$$
for all $m \gg 0$ because the $I$-filtration $\calF=\{M_n \}_{n \in \Z}$ of $M$ is stable so that $M_m \subseteq I^{n-q}M_q$ for all $m \gg 0$.
Thus, we have $M=M_0=M_q$, and whence $M_n=I^{n-q}M$ for all $n \in \Z$.
\end{proof}

Let $M$ be a finitely generated $A$-module and 
$$ \mathbb{F}_{\bullet}: \cdots \to F_i \overset{\partial_i}{\to} F_{i-1} \to \cdots \to F_1 \overset{\partial_1}{\to} F_0 \overset{\varepsilon}{\to} M \to 0   $$
be a minimal free resolution of the $A$-module $M$.
We then have the following.

\begin{lem}\label{9.7}
Let $\ell \in \Z$.
Suppose that $F_0$ and $M$ are $I$-filtered $A$-modules with respect to $\calF_0=\{I^{n-\ell}F_0\}_{n \in \Z}$ and $\calF=\{I^{n-\ell}M\}_{n \in \Z}$ respectively.
Then the following assertions hold true.
\begin{itemize}
\item[(1)] The $A$-linear map $\varepsilon: F_0 \to M$ is strict, 
\item[(2)] we have an epimorphism $\gr(\varepsilon):\gr(\calF_0) \to \gr(\calF)$ of graded $\gr_I(A)$-modules, and
\item[(3)] $\ker \gr(\varepsilon) \subseteq {\calM}\cdot \gr(\calF_0)$.
\end{itemize}
\end{lem}

\begin{proof}
It is easy to see that the $A$-linear map $\varepsilon$ is strict.
Therefore, by Lemma \ref{9.4}, we have an epimorphism $\gr(\varepsilon):\gr(\calF_0) \to \gr(\calF)$ of graded $\gr_I(A)$-modules.
Let us now look at the following commutative diagram
\[\begin{array}{cccc}
\gr(\calF_0) & \mapright{\gr(\varepsilon)} &  \gr(\calF) & \to  0\\
\mapdown{\varepsilon_1} & & \mapdown{\varepsilon_2} & \\
\gr(\calF_0)/{\calM}{\cdot}\gr(\calF_0) & \mapright{} & \gr(\calF)/{\calM}{\cdot}\gr(\calF) & \to  0 
\end{array}\]
of graded $\gr_I(A)$-modules, where $\varepsilon_1$ and $\varepsilon_2$ denote canonical maps and the rows are exact sequences.
Then because $\gr(\calF_0)/{\calM}{\cdot}\gr(\calF_0) \cong \gr(\calF)/{\calM}{\cdot}\gr(\calF)$, we get $\ker \gr(\varepsilon) \subseteq {\calM}\cdot \gr(\calF_0)$ as required.
\end{proof}

The following theorem is the key for our argument.

\begin{thm}\label{9.8}
Let
$$ \mathbb{L}_{\bullet}: \cdots \to L_i \to L_{i-1} \to \cdots \to L_1 \to L_0 \to \gr_I(M) \to 0 $$
be a minimal free resolution of the graded $\gr_I(A)$-module $\gr_I(M)$ and put $r_i={\rank}_{\gr_I(A)}L_i$ for all $i \geq 0$.
Assume that $L_i \cong \bigoplus^{r_i}\gr_I(A)(-a_i)$ as graded $\gr_I(A)$-modules for all $i \geq 0$ with $a_0=0$ and $a_i <a_j$, if $i < j$.
Suppose that $F_i$ is an $I$-filtered $A$-module with respect to ${{\calF}_i}=\{I^{n-a_i}F_i\}_{n \in \Z}$ for $i \geq 0$ and $M$ is an $I$-filtered $A$-module with respect to an $I$-adic filtration $\{I^nM\}_{n \in \Z}$ of $M$.
Then we have the following.
\begin{itemize}
\item[(1)] The $A$-linear maps $\varepsilon$ and $\partial_i$ are strict for all $i \geq 1$,
\item[(2)] the sequence
$$\cdots \to \gr(\calF_i) \overset{\gr(\partial_i)}{\to} \gr(\calF_{i-1}) \to \cdots \to \gr(\calF_1) \overset{\gr(\partial_1)}{\to} \gr(\calF_0) \overset{\gr(\varepsilon)}{\to} \gr_I(M) \to 0 $$
forms a minimal free resolution of the graded $\gr_I(A)$-module $\gr_I(M)$, and
\item[(3)] $A/I \otimes \partial_i=0$ for all $i \geq 1$.
\end{itemize}
\end{thm}

\begin{proof}
(1) and (2)
Let $\partial_0=\varepsilon$.
We want to show that the map $\partial_p$ is strict and the sequence
$$ \gr(\calF_p) \overset{\gr(\partial_p)}{\to} \gr(\calF_{p-1}) \to \cdots \to \gr(\calF_1) \overset{\gr(\partial_1)}{\to} \gr(\calF_0) \overset{\gr(\partial_0)}{\to} \gr_I(M) \to 0 $$
forms a part of a minimal free resolution of the graded $\gr_I(A)$-module $\gr_I(M)$ for all $p \geq 0$. 
We proceed by induction on $p$.

Suppose $p=0$ then the $A$-linear map $\partial_0:F_0 \to M$ is strict and the sequence 
$\gr(\partial_0) : \gr(\calF_0) \to \gr_I(M) \to 0$
forms a part of a minimal free resolution of the graded $\gr_I(A)$-module $\gr_I(M)$ by Lemma \ref{9.7}.

Assume that $p \geq 1$ and that $A$-linear maps $\partial_{i}$ are strict for all $0 \leq i \leq p-1$ and the sequence
$$ \gr(\calF_{p-1}) \overset{\gr(\partial_{p-1})}{\to} \gr(\calF_{p-2}) \to \cdots \to \gr(\calF_1) \overset{\gr(\partial_1)}{\to} \gr(\calF_0) \overset{\gr(\varepsilon)}{\to} \gr_I(M) \to 0 $$
forms a part of a minimal free resolution of the graded $\gr_I(A)$-module $\gr_I(M)$.
Let $Y=\ker \partial_{p-1} \subseteq F_{p-1}$ and put $Y_n=Y \cap I^{n-a_{p-1}}F_{p-1}$ for all $n \in \Z$.
Suppose that $Y$ is an $I$-filtered $A$-module with respect to $\calY=\{Y_n\}_{n \in \Z}$.
Then it is easy to check that the inclusion map $i_{p-1} : Y \hookrightarrow F_{p-1}$ is strict.
Hence the sequence 
$$ 0 \to \gr(\calY) \overset{\gr(i_{p-1})}{\to} \gr(\calF_{p-1}) \overset{\gr(\partial_{p-1})}{\to} \gr(\calF_{p-2})$$ 
is exact by Lemma \ref{9.4}.
Therefore we have $\gr(\calY)=\gr_I(A)\cdot[\gr(\calY)]_{a_{p}}$ because $L_p=\bigoplus^{r_p}\gr_I(A)(-a_p)$ by our assumption.
Then, thanks to Lemma \ref{9.6}, we have $Y_n=I^{n-a_{p}}Y$ for all $n \in \Z$ because the $I$-filtration $\calY=\{Y_n=Y \cap I^{n-a_{p-1}}F_{p-1}\}_{n \in \Z}$ of $Y$ is stable by the Artin-Rees Lemma.
Let $\tau_{p}: F_{p} \to Y$ be an $A$-linear map such that $\partial_{p}=i_{p-1} \circ \tau_{p}$.
Then the $A$-linear map $\tau_{p} : F_{p} \to Y$ is strict and we have an epimorphism $\gr(\tau_p):\gr(\calF_p) \to \gr(\calY)$ of graded $\gr_I(A)$-modules with $\ker \gr(\tau_p) \subseteq {\calM}\cdot \gr(\calF_p)$ by Lemma \ref{9.7}, because $F_p$ and $Y$ are $I$-filtered $A$-modules with respect to $\calF_p=\{ I^{n-a_p}F_p \}_{n \in \Z}$ and $\calY=\{I^{n-a_p}Y\}_{n \in \Z}$ respectively.
Thus the map $\partial_{p}: F_{p} \to F_{p-1}$ is strict and the sequence
$$ \gr(\calF_{p}) \overset{\gr(\partial_{p})}{\to} \gr(\calF_{p-1}) \overset{\gr(\partial_{p-1})}{\to} \gr(\calF_{p-2}) \to \cdots \to  \gr(\calF_{1}) \overset{\gr(\partial_{1})}{\to} \gr(\calF_0) \overset{\gr(\varepsilon)}{\to} \gr_I(M) \to 0 $$
forms a part of a minimal free resolution of the graded $\gr_I(A)$-module $\gr_I(M)$.
This completes the proof of assertions (1) and (2).

(3)
Since the $A$-linear maps $\partial_i$ are strict by our assertion (1) and $a_i>a_{i-1}$, we have 
$$ \partial_i(F_i)=\partial_i(I^{a_i-a_i}F_i)=\partial_i([F_i]_{a_i})=\partial_i(F_i) \cap [F_{i-1}]_{a_i} \subseteq [F_{i-1}]_{a_i}=I^{a_i-a_{i-1}}F_{i-1} \subseteq IF_{i-1} $$
for all $i \geq 1$ as required.
\end{proof}

The following Corollary \ref{9.9} shows that, for an Ulrich ideal $I$ of $A$, the residue class ring $A/I$ has a linear free resolution.

\begin{cor}\label{9.9}
Suppose that $A$ is a Cohen--Macaulay local ring and assume that $I$ is an Ulrich ideal of $A$.
Let
$$ \mathbb{F}_{\bullet}: \cdots \to F_i \overset{\partial_i}{\to} F_{i-1} \to \cdots \to F_1 \overset{\partial_1}{\to} F_0 \overset{\varepsilon}{\to} A/I \to 0 $$
be a minimal free resolution of the $A$-module $A/I$.
Suppose that $F_i$ is an $I$-filtered $A$-module with respect to ${\calF}_i=\{ I^{n-i}F_i \}_{n \in \Z}$ for $i \geq 1$ and $A/I$ is an $I$-filtered $A$-module with respect to an $I$-adic filtration $\{I^n(A/I)\}_{n \in \Z}$ of $A/I$.
Then we have the following.
\begin{itemize}
\item[(1)] The $A$-linear maps $\varepsilon$ and $\partial_i$ are strict for all $i \geq 1$ and
\item[(2)] the sequence
$$ \cdots \to \gr(\calF_i) \overset{\gr(\partial_i)}{\to} \gr(\calF_{i-1}) \to \cdots \to \gr(\calF_1) \overset{\gr(\partial_1)}{\to} \gr(\calF_0) \overset{\gr(\varepsilon)}{\to} \gr_I(A/I) \to 0 $$
forms a minimal free resolution of the graded $\gr_I(A)$-module $\gr_I(A/I)$.
\end{itemize}
\end{cor}

Therefore, for the associated graded ring $\gr_I(A)$ of an Ulrich ideal $I$ of $A$,
$$ \cdots \to \bigoplus^{\beta_i}\gr_I(A)(-i) \to \cdots \to \bigoplus^{\beta_1}\gr_I(A)(-1) \to \bigoplus^{\beta_0}\gr_I(A) \to \gr_I(A)/J \to 0 $$
forms a minimal free resolution of the graded $\gr_I(A)$-module $\gr_I(A)/J$, where $\beta_i$ denotes the $i$-th Betti number $\beta_i^A(A/I)$ of $A/I$ for $i \geq 0$.
We then have $\beta_0=1$ and $\beta_i=\binom{d}{i}+(\mu_A(I)-d){\cdot}\beta_{i-1}$ for $i \geq 1$ by Theorem \ref{7.1}.

In the proof of Corollary \ref{9.9}, we need the following lemma.

\begin{lem}\label{9.10}
Suppose that $A$ is a Cohen--Macaulay local ring with $d>0$.
Let $I$ be an $\m$-primary ideal of $A$ and let $Q$ be a parameter ideal of $A$ which forms a reduction of $I$.
Put $a \in Q \backslash \m Q$, $f=at \in R$, and $Z_i=\Syz_{\gr_I(A)}^i(\gr_I(A)/J)$ for $i \geq 0$.
Assume that $I/I^2$ is $A/I$-free and $f$ is a $\gr_I(A)$-regular element, then we have an isomorphism
$$ Z_i/fZ_i \cong \Syz_{\gr_I(A)/f{\cdot}\gr_I(A)}^{i-1}(\gr_I(A)/J)(-1)\oplus\Syz_{\gr_I(A)/f{\cdot}\gr_I(A)}^i(\gr_I(A)/J) $$
of graded $\gr_I(A)$-modules for all $i \geq 1$.
\end{lem}

\begin{proof}
Since $J/J^2 \cong (I/I^2)t$, $I/I^2$ is $A/I$-free, and $\gr_I(A)/J \cong A/I$, $J/J^2$ forms a finitely generated $\gr_I(A)/J$-free module.
Thus, by using the same technique as the proof of Claim in Lemma \ref{3.4}, we can prove $J/fJ \cong (\gr_I(A)/J)(-1)\oplus J/f{\cdot}\gr_I(A)$
as graded $\gr_I(A)$-modules, and whence we get the required assertion.
\end{proof}

\begin{proof}[Proof of Corollary $\ref{9.9}$]
We notice that we have $\gr_I(A/I) \cong \gr_I(A)/J$.
Let $$ \mathbb{L}_{\bullet}: \cdots \to L_i \to L_{i-1} \to \cdots \to L_1 \to L_0 \to \gr_I(A/I) \to 0 $$
be a minimal free resolution of the graded $\gr_I(A)$-module $\gr_I(A/I)$.
Put $Z_i=\Syz_{\gr_I(A)}^i(\gr_I(A/I))$ for $i \geq 0$ and $n=\mu_A(I)$.
Thanks to Theorem \ref{9.8}, we have only to show that $L_i \cong \bigoplus^{r_i}\gr_I(A)(-i)$ holds true for all $i \geq 0$, where $r_i=\rank_{\gr_I(A)}L_i$ denotes the $i$-th Betti number of $\gr_I(A/I)$.
We proceed by induction on $d$.
Suppose that $d=0$, then since $I^2=(0)$ and $I \cong (A/I)^n$ we have 
$J=It \cong (A/I)^nt \cong (\gr_I(A)/J)^n(-1).$
Therefore we get $Z_i \cong (\gr_I(A)/J)^{n^i}(-i)$ for all $i \geq 1$, inductively.
Hence $L_i=\bigoplus^{n_i}\gr_I(A)(-i)$ holds true for all $i \geq 1$.
Assume that $d>0$ and that our assertion holds true for $d-1$.
Let $a \in Q \backslash \m Q$ and put $f=at \in R$, $\ol{A}=A/(a)$, and $\ol{I}=I/(a)$.
Then $\ol{I}$ is an Ulrich ideal of $\overline{A}$ by Lemma \ref{3.3}.
Since $I^2=QI$ holds true, $f$ is a $\gr_I(A)$-regular element so that we have $\gr_I(A)/f{\cdot}\gr_I(A) \cong \gr_{\ol{I}}(\ol{A})$.
Hence, by Lemma \ref{9.10}, we have an isomorphism
$$ Z_i/fZ_i \cong \Syz_{\gr_{\ol{I}}(\ol{A})}^{i-1}(\gr_I(A)/J)(-1) \oplus \Syz_{\gr_{\ol{I}}(\ol{A})}^i(\gr_I(A)/J) $$
of graded $\gr_I(A)$-modules for all $i \geq 1$.
Then, by the hypothesis of induction on $d$, we have $\Syz_{\gr_{\ol{I}}(\ol{A})}^i(\gr_I(A)/J)=\gr_I(A){\cdot}[\Syz_{\gr_{\ol{I}}(\ol{A})}^i(\gr_I(A)/J)]_i$, whence $L_i \cong \bigoplus^{r_i}\gr_I(A)(-i)$ for all $i \geq 1$.
This completes the proof of Corollary \ref{9.9}.
\end{proof}

In my proof of Theorem \ref{9.5}, we need the following lemma.

\begin{lem}\label{9.11}
Let $I$ be an $\m$-primary ideal of $A$ and $Q=(a_1,a_2,\cdots,a_d)$ be a parameter ideal of $A$ which forms a reduction of $I$.
Put $f_i=a_it \in R$ for $1 \leq i \leq d$.
Suppose that $M$ is an Ulrich $A$-module with respect to $I$.
Then the following assertions hold true.
\begin{itemize}
\item[(1)] $I^nM=Q^nM$ for all $n \in \Z$,
\item[(2)] Suppose that $d=0$. Then $J{\cdot}{\gr}_I(M)=(0)$, $\gr_I(M)$ is $\gr_I(A)/J$-free, and ${\rank}_{\gr_I(A)/J}(\gr_I(M))=\mu_A(M)$.
\item[(3)] Suppose that $d>0$.
Then $f_1,f_2,\cdots,f_d$ forms a $\gr_I(M)$-regular sequence, $J{\cdot}\gr_I(M)=(f_1,f_2,\cdots,f_d)\gr_I(M)$, $\gr_I(M)/J{\cdot}\gr_I(M)$ is $\gr_I(A)/J$-free, and ${\rank}_{\gr_I(A)/J}(\gr_I(M)/J{\cdot}\gr_I(M))=\mu_A(M)$.
\end{itemize}
\end{lem}

\begin{proof}
(1)
We can prove the assertion by induction on $n$.

(2)
Because $I^nM=(0)$ for all $n >0$ by assertion (1), we have $J{\cdot}{\gr}_I(M)=(0)$.
Since $M \cong (A/I)^{\mu_A(M)}$ and $\gr_I(A)/J \cong A/I$, we have
$${\gr}_I(M)=[{\gr}_I(M)]_0 \cong M \cong (\gr_I(A)/J)^{\mu_A(M)}.$$
Therefore $\gr_I(M)$ is $\gr_I(A)/J$-free with $\rank_{\gr_I(A)/J}(\gr_I(M))=\mu_A(M)$.

(3)
Thanks to Valabrega-Valla's criterion (\cite{VV}, \cite[Theorem 1.1]{RV}), $f_1, f_2,\cdots, f_d$ forms a ${\gr}_I(M)$-regular sequence because $QM \cap I^{n+1}M=Q^{n+1}M$ holds true for all $n \geq 0$ by assertion (1).
It is easy to see that $J{\cdot}\gr_I(M)=(f_1,f_2,\cdots,f_d)\gr_I(M)$ holds true by assertion (1).
Because $M/IM \cong (A/I)^{\mu_A(M)}$ and $\gr_I(A)/J \cong A/I$, we have
$${\gr}_I(M)/J{\cdot}{\gr}_I(M) = [{\gr}_I(M)/J{\cdot}{\gr}_I(M)]_0 \cong M/IM \cong (\gr_I(A)/J)^{\mu_A(M)}.$$
Therefore $\gr_I(M)/J{\cdot}\gr_I(M)$ is $\gr_I(A)/J$-free with $\rank_{\gr_I(A)/J}(\gr_I(M)/J{\cdot}\gr_I(M))=\mu_A(M)$.
\end{proof}

Let us now give a proof of Theorem \ref{9.5}.

\begin{proof}[Proof of Theorem \ref{9.5}]
(1) and (2)
Let
\[
\mathbb{L}_{\bullet}:
 \cdots \to L_i \to L_{i-1} \to \cdots \to L_1 \to L_0 \to \gr_I(M) \to 0 
\]
denote a minimal free resolution of the graded $\gr_I(A)$-module $\gr_I(M)$ and 
put $Z_i=\Syz_{\gr_I(A)}^i(\gr_I(M))$ for all $i \geq 0$.
Thanks to Theorem \ref{9.8},
 we have only to show that $L_i \cong \bigoplus^{r_i}\gr_I(A)(-i)$ for all 
$i \geq 0$ where $r_i=\rank_{\gr_I(A)}L_i$.
We proceed by induction on $d$.

When $d=0$, we have $\gr_I(M) \cong (\gr_I(A)/J)^{\mu_A(M)}$ by Lemma \ref{9.11} (2).
Look at the exact sequence
$$ 0\to \bigoplus^{\mu_A(M)}J \to \gr_I(A)^{\mu_A(M)} \to (\gr_I(A)/J)^{\mu_A(M)} \to 0 $$
of $\gr_I(A)$-modules.
Since $J \cong (\gr_I(A)/J)^{\mu_A(I)}(-1)$ as in the proof of Corollary \ref{9.9}, we get $Z_i \cong (\gr_I(A)/J)^{\mu_A(M){\cdot}\mu_A(I)^i}(-i)$ for all $i \geq 1$, inductively.
Therefore we get $L_i \cong \bigoplus^{\mu_A(M){\cdot}\mu_A(I)^i}\gr_I(A)(-i)$ as required.

Assume that $d>0$ and that our assertion holds true for $d-1$.
Let $a \in Q \backslash \m Q$ and $f=at \in R$.
We put $\ol{A}=A/(a)$, $\ol{I}=I/(a)$, and $\ol{M}=M/aM$.
Then $\ol{I}$ is an Ulrich ideal of $\ol{A}$ by Lemma \ref{3.3}, and $\ol{M}$ is an Ulrich $\ol{A}$-module with respect to $\ol{I}$.
Then $f$ is $\gr_I(M)$-regular by Lemma \ref{9.11} (3) and hence we have $\gr_I(M)/f{\cdot}\gr_I(M) \cong \gr_{\ol{I}}(\ol{M})$.
The element $f$ is also $\gr_I(A)$-regular because $I^2=QI$ holds true.
Then we get an exact sequence
$$ \cdots \to L_i/fL_i \to L_{i-1}/fL_{i-1} \to \cdots \to L_1/fL_1 \to L_0/fL_0 \to \gr_{\ol{I}}(\ol{M}) \to 0$$
of graded $\gr_{\overline{I}}(\ol{A})$-modules.
Therefore, by the hypothesis of induction on $d$, we have $$Z_i/fZ_i \cong \Syz_{\gr(\ol{I})}^i(\gr_{\ol{I}}(\ol{M}))=\gr_I(A){\cdot}[\Syz_{\gr(\ol{I})}^i(\gr_{\ol{I}}(\ol{M})) ]_i.$$
Thus we get $L_i=\bigoplus^{r_i}\gr_I(A)(-i)$ for all $i \geq 1$.
Consequently our assertions (1) and (2) hold true by Theorem \ref{9.8}.

(3) and (4)
We notice that we have ${\rmI}_1(\partial_i) \subseteq I$ for all $i\geq 1$ and $\rank_AF_i=\rank_{\gr_I(A)}L_i$ for all $i \geq 0$ by Theorem \ref{9.8}.
We proceed by induction on $d$.
Suppose $d=0$.
Then the assertion is obvious by the proof of assertion (1) and (2).
Therefore the induction on $d$ easily shows that ${\rmI}_1(\partial_i)+Q=I$ for all $i \geq 1$ and
$$\rank_{\gr_I(A)}L_i=\rank_{\gr_{\ol{I}}(\ol{A})}{L_i/f L_i}=\mu_{\ol{A}}(\ol{M})\{\mu_{\ol{A}}(\ol{I})-\dim \ol{A}\}^i=\mu_A(M)\{\mu_A(I)-d\}^i$$
for all $i \geq 0$.
This completes the proof of assertions (3) and (4).
\end{proof}

We end this section by constructing an example of an Ulrich module with respect to an Ulrich ideal in the one dimensional case.
The assertion (3) of Example \ref{9.12} follows from Example \ref{7.3} and implies that the equality $\rmI_1(\partial_i)=I$ in Theorem \ref{9.5} (4) does not hold true in general.

\begin{ex}\label{9.12}
Let $A=k[[X,Y]]/(Y^2)$ where $k[[X,Y]]$ is the formal power series ring over a field $k$.
Put $\m=(x,y)$ where $x$ and $y$ denote the images of $X$ and $Y$ in $A$ respectively.
Let $I_n=(x^n,y)$ for $n \geq 1$.
Then $A$ is a Gorenstein local ring with $\dim A=1$ and we have the following.
\begin{itemize}
\item[(1)] $\mu_A(I_n)=2$ and $I_n$ is an Ulrich ideal of $A$ containing a reduction $(x^n)$ for $n \geq 1$.
\item[(2)] $I_n$ is an Ulrich $A$-module with respect to the maximal ideal $\m=I_1$ of $A$ for $n \geq 1$.
\item[(3)] The sequence
$$
\mathbb{F}_{\bullet} : \cdots \to A^2 \mapright{\begin{pmatrix}
-y& 0\\
x^n & y
\end{pmatrix}}
A^2 \mapright{\begin{pmatrix}
-y& 0\\
x^n & y
\end{pmatrix}}A^2 \mapright{\begin{pmatrix}
x^n & y
\end{pmatrix}} A \overset{\varepsilon}{\to} A/I_n \to 0. $$
forms a minimal free resolution of $A$-module $A/I_n$ for $n \geq 1$.
Therefore $\Syz_A^i(A/I_n) \cong I_n$ for all $i \geq 1$.
\end{itemize} 

\end{ex}

\section{Ulrich ideals of one-dimensional Gorenstein local rings of finite CM--representation type}\label{onefcm}

In Section \ref{minfree} we observed that every Cohen--Macaulay local ring of finite CM-representation type admits only finitely many nonparameter Ulrich ideals (Theorem \ref{7.7}).
In this section, we consider giving complete classification of those ideals, and do it for Gorenstein local rings of dimension one under some mild assumptions.
To achieve our purpose, we use techniques from the representation theory of maximal Cohen--Macaulay modules.
Let us begin with recalling several definitions and basic facts stated in Yoshino's book \cite{Y}.

\begin{defn}\cite[(2.8),(3.11) and (13.5)]{Y}\label{8.5}
Let $A$ be a $d$-dimensional Cohen--Macaulay complete local ring.
Suppose that $A$ is an {\em isolated singularity}, that is, the local ring $A_\fkp$ is regular for every nonmaximal prime ideal $\fkp$ of $A$.
Let $M$ be a nonfree indecomposable maximal Cohen--Macaulay $A$-module.
Then we define the {\em Auslander-Reiten translation} of $M$ by:
$$
\tau M=\Hom_A(\Syz_A^d(\operatorname{Tr}M),\rmK_A).
$$
Here $\rmK_A$ denotes the canonical module of $A$.
\end{defn}

\begin{lem}\label{8.6}
With the notation of Definition \ref{8.5}, assume that $A$ is Gorenstein with $d=1$.
Then one has an isomorphism $\tau M\cong\Syz_A^1(M)$.
\end{lem}

\begin{proof}
Since $M$ is nonfree and indecomposable, there exists an exact sequence
$$
\cdots \xrightarrow{\partial_2} F_1 \xrightarrow{\partial_1} F_0 \xrightarrow{\partial_0} F_{-1} \xrightarrow{\partial_{-1}} \cdots
$$
of finitely generated free $A$-modules whose $A$-dual is also exact such that $\Im\,\partial_i\subseteq\fkm F_{i-1}$ for all integers $i$ and $\Im\,\partial_0=M$.
We see from this exact sequence that $\tau M=(\Syz_A^1(\operatorname{Tr}M))^\ast\cong(\Im(\partial_1^\ast))^\ast\cong\Im\,\partial_1=\Syz_A^1(M)$.
\end{proof}

Let $A$ be a Cohen--Macaulay local ring.
The {\em Auslander-Reiten quiver} $\Gamma_A$ of $A$ is a graph consisting of vertices, arrows and dotted lines.
The vertices are the isomorphism classes of indecomposable maximal Cohen--Macaulay $A$-modules.
For nonfree indecomposable maximal Cohen--Macaulay modules $M$ and $N$, the vertex $[M]$ is connected by a dotted line with the vertex $[N]$ if and only if $M\cong\tau N$ and $N\cong\tau M$.
We refer to \cite[(5.2)]{Y} for details.
For a $1$-dimensional hypersurface, the Auslander-Reiten quiver finds all the pairs of maximal Cohen--Macaulay modules one of which is the first syzygy of the other:

\begin{prop}\label{8.7}
Let $A$ be a local hypersurface of dimension one.
Let $M,N$ be nonfree indecomposable maximal Cohen--Macaulay $A$-modules.
Then the following are equivalent.
\begin{enumerate}[\rm(1)]
\item
$M\cong\Syz_A^1(N)$.
\item
$N\cong\Syz_A^1(M)$.
\item
In $\Gamma_A$ the vertices $[M],[N]$ are connected by a dotted line.
\end{enumerate}
\end{prop}

\begin{proof}
Since $A$ is a hypersurface and $M,N$ are nonfree indecomposable, we have $M\cong\Syz_A^2(M)$ and $N\cong\Syz_A^2(N)$ (cf. \cite[(7.2)]{Y}).
By Lemma \ref{8.6}, we obtain the equivalence.
\end{proof}

Throughout the rest of this section, let $A$ be a $1$-dimensional Gorenstein local ring.
We denote by $\calC_A$ the set of nonisomorphic maximal Cohen--Macaulay $A$-modules $M$ without nonzero free summand such that $\Syz_A^1(M)\cong M$ and $\mu_A(M)=2$.
The following statement relates the notion of Ulrich ideals with the representation theory of maximal Cohen--Macaulay modules.

\begin{prop}\label{8.8}
Let $A$ be a $1$-dimensional Gorenstein local ring.
Then one has the inclusion $\calX_A\subseteq\calC_A$.
\end{prop}

\begin{proof}
If an ideal $I$ of $A$ has a nonzero free summand, then we can write $I=(x)\oplus J$ for some nonzerodivisor $x$ and ideal $J$ of $A$.
Since $xJ\subseteq(x)\cap J=(0)$, we have $J=(0)$, and $I=(x)$.
Thus every ideal that is an element of $\X_A$ does not have a nonzero free summand.
The assertion now follows from Corollaries \ref{2.6} and \ref{7.4}.
\end{proof}

Let $A$ be a $1$-dimensional Gorenstein complete equicharacteristic local ring with algebraically closed residue field $k$ of characteristic $0$.
Suppose that $A$ has finite CM-representation type.
Then $A$ is a {\em simple singularity}, namely, one has a ring isomorphism
$$
A\cong k[[x,y]]/(f),
$$
where $f$ is one of the following:
\begin{gather*}
(\rmA_n)\ x^2+y^{n+1} \quad(n\ge1),\qquad(\rmD_n)\ x^2y+y^{n-1} \quad (n\ge4),\\
(\rmE_6)\ x^3+y^4,\qquad(\rmE_7)\ x^3+xy^3,\qquad(\rmE_8)\ x^3+y^5.
\end{gather*}
For the details, see \cite[(8.5), (8.10) and (8.15)]{Y}.
In this case, we can make a complete list of the nonparameter Ulrich ideals.

\begin{thm}\label{8.9}
With the above notation, the set $\X_A$ is equal to:
\begin{enumerate}
\item[$(\rmA_n)$\ ]
$\begin{cases}
\{(x,y),(x,y^2),\dots,(x,y^{\frac{n}{2}})\} & \text{if $n$ is even},\\
\{(x,y),(x,y^2),\dots,(x,y^{\frac{n-1}{2}}),(x,y^{\frac{n+1}{2}})\} & \text{if $n$ is odd}.
\end{cases}$
\item[$(\rmD_n)$\ ]
$\begin{cases}
\{(x^2,y),(x+\sqrt{-1}y^\frac{n-2}{2},y^\frac{n}{2}),(x-\sqrt{-1}y^\frac{n-2}{2},y^\frac{n}{2})\} & \text{if $n$ is even},\\
\{(x^2,y),(x,y^\frac{n-1}{2})\} & \text{if $n$ is odd}.
\end{cases}$
\item[$(\rmE_6)$\ ]
$\{(x,y^2)\}$.
\item[$(\rmE_7)$\ ]
$\{(x,y^3)\}$.
\item[$(\rmE_8)$\ ]
$\emptyset$.
\end{enumerate}
\end{thm}

\begin{proof}
Thanks to Proposition \ref{8.8}, the set $\X_A$ is contained in $\calC_A$, so it is essential to calculate $\calC_A$.
It is possible by looking at the Auslander-Reiten quiver $\Gamma_A$ of $A$, which is described in \cite{Y}.
More precisely, by virtue of Proposition \ref{8.7}, all elements of $\calC_A$ are direct sums of modules corresponding to vertices of $\Gamma_A$ connected by dotted lines.
Once we get the description of $\calC_A$, we can find elements of $\calC_A$ belonging to $\X_A$, by making use of Corollary \ref{2.6}.

(1) The case $(\rmA_n)$ with $n$ even:

It follows from \cite[(5.11) and (5.12)]{Y} that
$$
\calC_A=\{(x,y),(x,y^2),\dots,(x,y^\frac{n}{2})\}.
$$
Applying Corollary \ref{2.6} to $I=(x,y^i)$ and $Q=(y^i)$ for $1\le i\le\frac{n}{2}$, we see that $(x,y^i)$ is an Ulrich ideal.
Hence $\calX_A=\calC_A=\{(x,y),(x,y^2),\dots,(x,y^\frac{n}{2})\}$.

(2) The case $(\rmA_n)$ with $n$ odd:

We use the same notation as in \cite[(9.9)]{Y}.
It is seen by \cite[Figure (9.9.1)]{Y} that
$$
\calC_A\subseteq\{M_1,M_2,\dots,M_\frac{n-1}{2},N_+\oplus N_-\}
$$
holds.
For $1\le j\le n+1$, the sequence
$$
\begin{CD}
A^2 @>{\left(\begin{smallmatrix}
x & y^j \\
y^{n+1-j} & -x
\end{smallmatrix}\right)}>> A^2 @>{\text{nat}}>> (x,y^j) @>>> 0
\end{CD}
$$
is exact, which shows that $M_j$ is isomorphic to the ideal $(x,y^j)$ of $A$.
Since $N_+\oplus N_-\cong M_\frac{n+1}{2}$, we have $\calC_A=\{(x,y),(x,y^2),\dots,(x,y^{\frac{n-1}{2}}),(x,y^{\frac{n+1}{2}})\}$.
Applying Corollary \ref{2.6} to $I=(x,y^j)$ and $Q=(y^j)$ yields that $(x,y^j)$ is an Ulrich ideal for $1\le j\le\frac{n+1}{2}$.
Therefore $\calX_A=\calC_A=\{(x,y),(x,y^2),\dots,(x,y^{\frac{n-1}{2}}),(x,y^{\frac{n+1}{2}})\}$.

(3) The case $(\rmD_n)$ with $n$ odd:

We adopt the same notation as in \cite[(9.11)]{Y}, except that we use $A'$ instead of $A$ there.
By \cite[Figure (9.11.5)]{Y} we have the inclusion relation
$$
\calC_A\subseteq\{A'\oplus B,X_1\oplus Y_1,M_1\oplus N_1,X_2\oplus Y_2,\dots,M_\frac{n-3}{2}\oplus N_\frac{n-3}{2},X_\frac{n-1}{2}\}.
$$
Taking into account the minimal number of generators, we observe that $\calC_A=\{A'\oplus B,X_\frac{n-1}{2}\}$.
Since $(0):y=(x^2+y^{n-2})$, $(0):(x^2+y^{n-2})=(y)$ and $(x^2+y^{n-2})\cap (y)=(0)$, we have
$$
A'\oplus B=A/(y)\oplus A/(x^2+y^{n-2})\cong (x^2+y^{n-2})\oplus (y)=(x^2+y^{n-2},y)=(x^2,y).
$$
As $X_\frac{n-1}{2}\cong Y_\frac{n-1}{2}\cong (x,y^\frac{n-1}{2})$, we get $\calC_A=\{(x^2,y),(x,y^\frac{n-1}{2})\}$.
Put $I=(x^2,y)\supseteq Q=(x^2-y)$.
Then $QI=(x^4+y^{n-1},y^2(1+y^{n-3}))=(x^4,y^2)=I^2$, since $1+y^{n-3}\in A$ is a unit as $n\ge4$.
We see that $A/Q$ is Artinian, whence $Q$ is a parameter ideal of $A$.
It is straightforward that $Q:I=I$ holds, and Corollary \ref{2.6} shows that $(x^2,y)$ is an Ulrich ideal.
Also, using Corollary \ref{2.6} for $I:=(x,y^\frac{n-1}{2})\supseteq Q:=(x)$, we observe that $(x,y^\frac{n-1}{2})$ is an Ulrich ideal.
Thus, we obtain $\calX_A=\calC_A=\{(x^2,y),(x,y^\frac{n-1}{2})\}$.

(4) The case $(D_n)$ with $n$ even:

We adopt the same notation as in \cite[(9.12)]{Y}, except that we use $A'$ instead of $A$ there.
It follows from \cite[Figure (9.12.1)]{Y} that
$$
\calC_A\subseteq\{ A'\oplus B,X_1\oplus Y_1,M_1\oplus N_1,X_2\oplus Y_2,\dots,X_\frac{n-2}{2}\oplus Y_\frac{n-2}{2},C_+\oplus D_+,C_-\oplus D_-\}.
$$
Restricting to the modules generated by at most two elements, we have $\calC_A=\{ A'\oplus B,C_+\oplus D_+,C_-\oplus D_-\}$.
Similarly to (3), we get isomorphisms $A'\oplus B \cong(x^2,y)$ and $C_\pm\oplus D_\pm\cong(y^\frac{n}{2},x\mp\sqrt{-1}y^\frac{n-2}{2})$.
Hence $\calC_A=\{(x^2,y),(y^\frac{n}{2},x-\sqrt{-1}y^\frac{n-2}{2}),(y^\frac{n}{2},x+\sqrt{-1}y^\frac{n-2}{2})\}$.
We have $(x^2,y)\in\calX_A$ similarly to (3).

Let us consider the ideal $I=(y^\frac{n}{2},x-\sqrt{-1}y^\frac{n-2}{2})$.
Set $Q=((x-\sqrt{-1}y^\frac{n-2}{2})+y(x+\sqrt{-1}y^\frac{n-2}{2}))$.
To check that $I,Q$ satisfy the assumptions of Corollary \ref{2.6}, we apply the change of variables $x\mapsto\sqrt{-1}x,\ y\mapsto y$ and put $n=2m+2$ with $m\ge1$.
We may assume:
$$
A=k[[x,y]]/(x^2y-y^{2m+1}),\ I=(y^{m+1},x-y^m),\ Q=((x-y^m)+y(x+y^m)).
$$
Note that $xy^{m+1}-y^{2m+1}=\frac{1}{2}(x^2+y^{2m})y-y^{2m+1}=0$ in the residue ring $A/(x-y^m)^2$.
Hence $I^2=(y^{2m+2},(x-y^m)^2)$ and $QI=((x-y^m)y^{m+1}+y^{m+2}(x+y^m),(x-y^m)^2)=(2y^{2m+2},(x-y^m)^2)$, from which $I^2=QI$ follows.
Clearly, $I$ contains $Q$.
We have
\begin{align*}
A/Q & = k[[x,y]]/(x^2y-y^{2m+1},(1+y)x-(1-y)y^m) \\
& = k[[x,y]]/(((1+y)x)^2y-(1+y)^2y^{2m+1},(1+y)x-(1-y)y^m) \\
& = k[[x,y]]/(((1-y)y^m)^2y-(1+y)^2y^{2m+1},x-(1+y)^{-1}(1-y)y^m) \\
& = k[[x,y]]/(-4y^{2m+2},x-(1+y)^{-1}(1-y)y^m) \cong k[[y]]/(y^{2m+2}).
\end{align*}
This especially says that $Q$ is a parameter ideal, and the isomorphism corresponds $I/Q=y(y^m,x)A/Q$ to $y^{m+1}k[[y]]/(y^{2m+2})$.
Hence $(Q:_AI)/Q=(0:_{A/Q}I/Q)=I/Q$, and therefore $Q:I=I$.
Now we can apply Corollary \ref{2.6}, and see that $I$ is an Ulrich ideal.

The change of variables $x\mapsto-x,\ y\mapsto y$ shows that $(y^\frac{n}{2},x+\sqrt{-1}y^\frac{n-2}{2})$ is an Ulrich ideal.
Thus $\calX_A=\calC_A=\{(x^2,y),(y^\frac{n}{2},x-\sqrt{-1}y^\frac{n-2}{2}),(y^\frac{n}{2},x+\sqrt{-1}y^\frac{n-2}{2})\}$.

(5) The case $(\rmE_6)$:

We adopt the same notation as in \cite[(9.13)]{Y}, except that we use $A'$ instead of $A$ there.
By \cite[Figure (9.13.1)]{Y} we have
$$
\calC_A\subseteq\{M_2,X,A'\oplus B,M_1\oplus N_1\}.
$$
We observe that $\mu_A(A'\oplus B)=6$, $\mu_A(M_1\oplus N_1)=\mu_A(X)=4$ and $M_2\cong (x^2,y^2)\cong(x^3,xy^2)=(y^4,xy^2)\cong(x,y^2)$.
Hence $\calC_A=\{ M_2\}=\{ (x,y^2)\}$.
Applying Corollary \ref{2.6} to $I=(x,y^2)$ and $Q=(x)$, we get $\calX_A=\calC_A=\{(x,y^2)\}$.

(6) The case $(\rmE_7)$:

We adopt the same notation as in \cite[(9.14)]{Y}, except that we use $A'$ instead of $A$ there.
According to \cite[Figure (9.14.1)]{Y},
$$
\calC_A\subseteq\{ A'\oplus B,C\oplus D,M_1\oplus N_1,M_2\oplus N_2,X_1\oplus Y_1,X_2\oplus Y_2,X_3\oplus Y_3\}
$$
holds.
We see that $\mu_A(C\oplus D)=\mu_A(M_1\oplus N_1)=\mu_A(M_2\oplus N_2)=4$, $\mu_A(X_1\oplus Y_1)=\mu_A(X_2\oplus Y_2)=6$, $\mu_A(X_3\oplus Y_3)=8$ and $A'\oplus B\cong(x,y^3)$.
Hence $\calC_A=\{(x,y^3)\}$.
Using Corollary \ref{2.6} for $I=(x,y^3)$ and $Q=(x-y^3)$, we get $I\in\X_A$.
Therefore $\calX_A=\calC_A=\{(x,y^3)\}$.

(7) The case $(\rmE_8)$:

We use the same notation as in \cite[(9.15)]{Y}.
By \cite[Figure (9.15.1)]{Y} we have
$$
\calC_A\subseteq\{ A_1\oplus B_1,A_2\oplus B_2,C_1\oplus D_1,C_2\oplus D_2,M_1\oplus N_1,M_2\oplus N_2,X_1\oplus Y_1,X_2\oplus Y_2\}.
$$
We have $\mu_A(M_i\oplus N_i)=4$, $\mu_A(A_i\oplus B_i)=6$ and $\mu_A(C_i\oplus D_i)=8$ for $i=1,2$, and have $\mu_A(X_1\oplus Y_1)=12$ and $\mu_A(X_2\oplus Y_2)=10$.
Consequently, we get $\calX_A=\calC_A=\emptyset$.
\end{proof}

The proof of Theorem \ref{8.9} yields the following result.

\begin{cor}\label{8.10}
Let $A$ be a $1$-dimensional complete equicharacteristic Gorenstein local ring with algebraically closed residue field of characteristic $0$.
If $A$ has finite CM-representation type, then one has $\calX_A=\calC_A$.
\end{cor}

\begin{rem}
Without the assumption that $A$ has finite CM-representation type, the equality in Corollary \ref{8.10} does not necessarily hold true even if $A$ is a $1$-dimensional complete intersection (cf. Remark \ref{2.8}).
\end{rem}

\section*{Acknowledgments}
The authors thank the referee for helpful comments.

\end{document}